\documentclass[a4paper]{amsart}
\usepackage{amsmath,amsthm,amssymb,latexsym,epic,bbm,comment}
\usepackage{graphicx,enumerate,stmaryrd, color}
\usepackage[all,2cell]{xy}
\xyoption{2cell}

\newtheorem{theorem}{Theorem}
\newtheorem{lemma}[theorem]{Lemma}

\newtheorem{proposition}[theorem]{Proposition}

\newtheorem{example}[theorem]{Example}

\newtheorem{remark}[theorem]{Remark}
\usepackage[all]{xy}
\usepackage[active]{srcltx}
\usepackage[parfill]{parskip}
\usepackage{enumerate}

\newcommand{\mone}{{\,\text{-}1}}

\font\sc=rsfs10
\newcommand{\cC}{\sc\mbox{C}\hspace{1.0pt}}
\newcommand{\cG}{\sc\mbox{G}\hspace{1.0pt}}

\newcommand{\cA}{\sc\mbox{A}\hspace{1.0pt}}

\newcommand{\cH}{\sc\mbox{H}\hspace{1.0pt}}
\font\scc=rsfs7
\newcommand{\ccC}{\scc\mbox{C}\hspace{1.0pt}}

\begin{document}
\title[$2$-categories of symmetric bimodules and their $2$-representations]
{$2$-categories of symmetric bimodules\\ and their $2$-representations}

\author[V.~Mazorchuk, V.~Miemietz and X.~Zhang]
{Volodymyr Mazorchuk, Vanessa Miemietz and Xiaoting Zhang}

\begin{abstract}
In this article we analyze the structure of $2$-categories of 
symmetric projective bimodules over a finite dimensional algebra
with respect to the action of a finite abelian group. We determine
under which condition the resulting $2$-category is fiat (in the
sense of \cite{MM1}) and classify simple transitive $2$-representations
of this $2$-category (under some mild technical assumption).
We also study several classes of examples in detail.
\end{abstract}

\maketitle

\section{Introduction and description of the results}\label{s1}

In the last 20 years, many exciting breakthroughs in representation theory,
see e.g. \cite{Kh,CR,BS,EW,Wi}, have originated from the idea of categorification. 
This has inspired the subject of $2$-representation theory which studies 
$2$-categories with suitable finiteness conditions. An appropriate $2$-analogue 
of a finite dimensional algebra was defined in \cite{MM1} and called finitary 
$2$-category.  Various aspects of the structure and $2$-representation theory
of finitary and, more specifically, fiat $2$-categories have been studied
in \cite{MM1,MM2,MM3,MM4,MM5,MM6,MMMT,MMMZ,CM,ChMi}, see also references therein.
In particular, \cite{MM5} introduces the notion of simple transitive $2$-representation
which is an appropriate categorification of the concept of an irreducible representation.
A natural and interesting problem is the classification of simple transitive $2$-representations
for various classes of $2$-categories, see \cite{Ma}. 

One interesting example of a fiat $2$-category is the $2$-category of Soergel bimodules
associated to the coinvariant algebra of a finite Coxeter system, see \cite{So1,So2,EW}.
For these $2$-categories, simple transitive $2$-representations have been classified
in several special cases including type $A$, the dihedral types and some small classical
types, see \cite{MM5,KMMZ,MT,MMMZ}. The article \cite{MMMZ} develops a reduction technique
that reduces the classification problem to smaller $2$-categories which, in practice, 
are often given by ``symmetric bimodules'' as defined in a special case in \cite{KMMZ}.

Inspired by this, in the present article we formalize the concept of symmetric bimodules
and study the resulting $2$-categories of projective symmetric bimodules and their 
$2$-representations. We show that these $2$-categories are weakly fiat provided that
the underlying algebra is self-injective and fiat if the underlying algebra is
weakly symmetric, mirroring the situation for the $2$-category 
of all projective bimodules from \cite[Subsection~7.3]{MM1} and \cite[Subsection~2.8]{MM6}.
Using \cite{MMMZ}, we reduce the classification of simple transitive $2$-representations
in the weakly symmetric case to the classification of module categories over the
$2$-category $\mathrm{Rep}(G)$ from \cite{Os}.

One of the main results of \cite{MM3} classifies a class of ``simple'' $2$-categories
with a particularly nice combinatorial structure. Here we study one of the smallest
families of $2$-categories which do not fit into the setup of \cite{MM3}. We show
that these can always be realized inside a $2$-category of symmetric bimodules.

The paper is organized as follows. In Section~\ref{s2} we collect the necessary 
preliminaries. In Section~\ref{s3} we introduce symmetric bimodules and study 
their structure and simple transitive $2$-representations. The latter are classified
in Theorem~\ref{thm11}. In Section~\ref{s7} we study $2$-categories with one object
which, apart from the identity $1$-morphism, have precisely two indecomposable 
$1$-morphisms up to isomorphism and these two form a biadjoint pair and their own
left/right/two-sided cell. In Theorem~\ref{prop41} we realize such $2$-categories
as $2$-subcategories of certain symmetric bimodules. 
\vspace{0.5cm}

{\bf Acknowledgements.}
The first author is supported by the Swedish Research Council,
G{\"o}ran Gustafsson Foundation and Vergstiftelsen. 
The second author is supported by EPSRC grant EP/S017216/1.
The third author is supported by G{\"o}ran Gustafsson Foundation.
\vspace{1cm}

\section{Preliminaries}\label{s2}

\subsection{Setup}\label{s2.1}

We work over an algebraically closed field $\Bbbk$.

\subsection{Finitary $2$-categories}\label{s2.2}

A $\Bbbk$-linear category is called {\em finitary} if it is small and 
equivalent to the category of finitely generated projective modules over
some finite dimensional associative $\Bbbk$-algebra.

We call a $2$-category $\cC$  {\em finitary}  (over $\Bbbk$) if
it has finitely many objects, each morphism category $\cC(\mathtt{i},\mathtt{j})$ is a 
finitary $\Bbbk$-linear category, all compositions are (bi)additive and 
$\Bbbk$-linear and all identity $1$-morphisms are indecomposable (cf. \cite[Subsection~2.2]{MM1}).  

We say that $\cC$ is {\em weakly fiat} if it is finitary and has a 
weak antiautomorphism $(\,\,)^{\star}$ of finite order and adjunction morphisms, 
see \cite[Subsection~2.5]{MM6}. If $(\,\,)^{\star}$ is involutive, we say that 
$\cC$ is {\em fiat}, see \cite[Subsection~2.4]{MM1}.

\subsection{$2$-representations}\label{s2.3}

Let $\cC$ be a finitary $2$-category.
A {\em finitary $2$-representation} of $\cC$ is a $2$-functor from $\cC$
to the $2$-category $\mathfrak{A}_{\Bbbk}^{f}$ whose
\begin{itemize}
\item objects are finitary $\Bbbk$-linear categories;
\item $1$-morphisms are additive $\Bbbk$-linear functors;
\item $2$-morphisms are natural transformations of functors.
\end{itemize}
Such $2$-representations form a $2$-category denoted $\cC\text{-}\mathrm{afmod}$, 
see \cite[Subsection~2.3]{MM3}. 

Similarly we define an {\em abelian $2$-representation} of $\cC$ 
as a $2$-functor from $\cC$ to the $2$-category whose
\begin{itemize}
\item objects are categories equivalent to categories of finitely generated modules 
over finite dimensional $\Bbbk$-algebras;
\item $1$-morphisms are right exact $\Bbbk$-linear functors;
\item $2$-morphisms are natural transformations of functors.
\end{itemize}
Such $2$-representations form a $2$-category denoted $\cC\text{-}\mathrm{mod}$, 
see \cite[Subsection~2.3]{MM3}. Following \cite[Subsection~4.2]{MM2}, we denote by
$\overline{\,\,\cdot\,\,}$ the {\em abelianization} $2$-functor from $\cC\text{-}\mathrm{afmod}$
to $\cC\text{-}\mathrm{mod}$.

A finitary $2$-representation $\mathbf{M}$ is called {\em simple transitive}
provided that $\displaystyle\coprod_{\mathtt{i}\in\ccC}\mathbf{M}(\mathtt{i})$ has no 
non-trivial $\cC$-invariant ideals, see \cite[Subsection~3.5]{MM5}.

\subsection{Cells and cell $2$-representations}\label{s2.21}

Given two indecomposable $1$-morphisms $\mathrm{F}$ and $\mathrm{G}$ in $\cC$, we define
$\mathrm{F}\geq_L\mathrm{G}$ if $\mathrm{F}$ is isomorphic to a direct summand
of $\mathrm{H}\circ\mathrm{G}$, for some $1$-morphism $\mathrm{H}$. This produces the
{\em left} preorder  $\geq_L$, of which the equivalence classes are called {\em left cells}.
Similarly one obtains the {\em right} preorder $\geq_R$ and the corresponding
{\em right cells}, and the {\em two-sided} preorder $\geq_J$ and the corresponding
{\em two-sided cells}.

For each simple transitive $2$-representation $\mathbf{M}$, there is a unique
two-sided cell which is maximal, with respect to $\geq_J$, among those two-sided
cells whose $1$-morphisms are not annihilated by $\mathbf{M}$. This 
two-sided cell is called the {\em apex} of $\mathbf{M}$, see \cite[Subsection~3.2]{CM}.

If $\mathcal{J}$ is a two-sided cell of $\cC$, we say that $\cC$ is {\em $\mathcal{J}$-simple}
if any non-zero $2$-ideal of $\cC$ contains the identity $2$-morphisms of all $1$-morphisms
in $\mathcal{J}$.

Each left cell of a fiat $2$-category contains a so-called {\em Duflo involution},
see \cite[Subsection~4.5]{MM1}.

For a left cell $\mathcal{L}$ in $\cC$, there exists $\mathtt{i}\in \cC$ such that all 
$1$-morphisms in $\mathcal{L}$ starts at $\mathtt{i}$. Denote by $\mathbf{N}_{\mathcal{L}}$ 
the $2$-subrepresentation of $\mathbf{P}_{\mathtt{i}}$ which is defined as
for each $\mathtt{j}\in\cC$ the category $\mathbf{N}_{\mathcal{L}}(\mathtt{j})$ is the 
additive closure of $\mathbf{P}_{\mathtt{i}}(\mathtt{j})$ consisting of all $1$-morphisms $\mathrm{F}$ 
with $\mathrm{F}\geq_{\mathcal{L}}\mathcal{L}$. From \cite[Lemma~3]{MM5}, we know that the 
$2$-representation $\mathbf{N}_{\mathcal{L}}$ contains a unique maximal ideal which does not 
contain any $\mathrm{id}_{\mathrm{F}}$ for $\mathrm{F}\in\mathcal{L}$, denoted $\mathcal{I}_{\mathcal{L}}$.
The quotient $\mathbf{C}_{\mathcal{L}}:=\mathbf{N}_{\mathcal{L}}/\mathcal{I}_{\mathcal{L}}$
is called the {\em cell $2$-representation} associated to $\mathcal{L}$.

\section{Symmetric bimodules and their simple transitive $2$-representations}\label{s3}

\subsection{Symmetric bimodules}\label{s3.1}

Let $A$ be a finite dimensional, unital, associative $\Bbbk$-algebra.
We assume that $A$ is basic and that $\{e_1,e_2,\dots,e_k\}$ is a complete
set of pairwise orthogonal primitive idempotents in $A$.

Let $G$ be a finite abelian subgroup of the group of automorphisms of $A$.
Assume that $\mathrm{char}(\Bbbk)$ does not divide $|G|$.
The action of $G$ on $A$ induces an action of $G$ on the category of $A$-$A$-bimodules via
$M\mapsto {}^{\varphi}M^{\varphi^{}}$, where the action of $A$ on ${}^{\varphi}M^{\varphi^{}}$
is given by
\begin{displaymath}
a\cdot m\cdot b:= \varphi(a)m\varphi^{}(b), \quad \text{ for all }a,b\in A\text{ and }m\in M.
\end{displaymath}
We will write ${}^{\varphi}f^{\varphi^{}}$ for the translate of a morphism $f$ under the action of
$\varphi\in G$.

Let $\mathcal{X}$ denote the category whose objects are $A$-$A$-bimodules and morphisms between
$A$-$A$-bimodules $M$ and $N$ are defined by
\begin{displaymath}
\mathrm{Hom}_{\mathcal{X}}(M,N):=
\bigoplus_{\varphi\in G}\mathrm{Hom}_{A\text{-}A}(M,{}^{\varphi}N^{\varphi^{}}).
\end{displaymath}
An element $f\in \mathrm{Hom}_{\mathcal{X}}(M,N)$ is thus represented by a tuple
$(f_{\varphi})_{\varphi\in G}$, where the component $f_{\varphi}$ is in
$\mathrm{Hom}_{A\text{-}A}(M,{}^{\varphi}N^{\varphi^{}})$.
For any $f\in \mathrm{Hom}_{\mathcal{X}}(M,N)$ and $g\in  \mathrm{Hom}_{\mathcal{X}}(N,K)$, considering
\begin{equation*}
\begin{array}{ccc}
\mathrm{Hom}_{A\text{-}A}(N,{}^{\psi}K^{\psi^{}})\otimes\mathrm{Hom}_{A\text{-}A}(M,{}^{\varphi}N^{\varphi^{}})
&\to& \mathrm{Hom}_{A\text{-}A}(M,{}^{\varphi\psi}K^{\varphi^{}\psi^{}})\\
g_{\psi}\otimes f_{\varphi}&\mapsto& {}^{\varphi}(g_{\psi})^{\varphi^{}}\circ f_{\varphi},
\end{array}
\end{equation*}
where we use $\varphi^{}\psi^{}=\psi^{}\varphi^{}$ on the right hand side,
the composition $g\circ f$ is given by
\begin{equation}\label{eq1}
\begin{array}{ccc}
\mathrm{Hom}_{\mathcal{X}}(N,K)\otimes\mathrm{Hom}_{\mathcal{X}}(M,N)
&\to& \mathrm{Hom}_{\mathcal{X}}(M,K)\\
(g_{\psi})_{\psi\in G}\otimes (f_{\varphi})_{\varphi\in G}&\mapsto& \big(\displaystyle\sum_{\varphi\in G}{}^{\varphi}(g_{\sigma\varphi^{\mone}})^{\varphi}\circ f_{\varphi}\big)_{\sigma\in G}.
\end{array}
\end{equation}
This composition can be depicted by the diagram
\begin{displaymath}
\xymatrix{M\ar[rrr]^{(f_{\varphi})_{\varphi\in G}}&&&\displaystyle\bigoplus_{\varphi\in G} {}^{\varphi}N^{\varphi}\ar[rrr]^{({}^{\varphi}(g_{\sigma\varphi^{\mone}})^{\varphi})_{\varphi,\sigma\in G}}&&&\displaystyle\bigoplus_{\sigma\in G} {}^{\sigma}K^{\sigma}}.
\end{displaymath}
We refer to \cite{CiMa} for details.

We denote by $\tilde{{\mathcal{X}}}$ the idempotent completion of $\mathcal{X}$, i.e. an object of $\tilde{{\mathcal{X}}}$ is given by a pair $(M,e)$ where $M$ is an $A$-$A$-bimodule and $e$ is an idempotent in $\mathrm{End}_{\mathcal{X}}(M)$. For an $A$-$A$-bimodule $M$, set
\begin{displaymath}
G_M:=\{\varphi\in G\,\vert\, M\cong {}^{\varphi}M^{\varphi^{}}\}
\end{displaymath}
which is a subgroup of $G$.

\begin{remark}\label{rem001}
{\em
As we will often encounter and use in this article, computation of homomorphism in
$\tilde{\mathcal{X}}$ using homomorphisms in $\mathcal{X}$ requires care.
Given $M$ and $M'$ in $\mathcal{X}$ and idempotents $e$ and $e'$ in
$\mathrm{End}_{\mathcal{X}}(M)$ and $\mathrm{End}_{\mathcal{X}}(M')$, respectively,
in the computation of $\mathrm{Hom}_{\tilde{\mathcal{X}}}((M,e),(M',e'))$
using $\mathrm{Hom}_{\mathcal{X}}(M,M')$
it is very important to make sure that the elements from $\mathrm{Hom}_{\mathcal{X}}(M,M')$
one works with do belong to $e'\circ \mathrm{Hom}_{\mathcal{X}}(M,M')\circ e$.
This is usually achieved by pre- and postcomposing the elements one works
with $e$ and $e'$, respectively.
Moreover, for any element $\boldsymbol{f}$ in $\mathrm{Hom}_{\tilde{\mathcal{X}}}((M,e),(M',e'))$,
we have
\begin{equation}\label{eq00}
\boldsymbol{f}\circ \text{id}_{(M,e)}=\boldsymbol{f}=\text{id}_{(M',e')}\circ\boldsymbol{f},
\end{equation}
where, in fact, $\text{id}_{(M,e)}=e$ and $\text{id}_{(M',e')}=e'$.
}
\end{remark}

As usual, we denote by $\hat{G}$ the {\em Pontryagin dual} of $G$ whose elements are
all group homomorphisms from $G$ to $\Bbbk^{*}$ with respect to point-wise multiplication.
As $G$ is finite and abelian, the group $\hat{G}$ is (non-canonically) isomorphic to $G$ and
$\hat{G}$ is canonically isomorphic to the group of isomorphism classes of simple $G$-modules
with respect to taking tensor products.

The group algebra $\Bbbk[G]$ is commutative and semi-simple and admits a unique decomposition
into a product of $|G|$ copies of $\Bbbk$. Let $\{\pi_{\chi},\chi\in \hat{G}\}$ be the corresponding
primitive idempotents. Each $\pi_{\chi}$ has the form
{\small$\displaystyle\frac{1}{|G|}\sum_{\alpha\in G}\chi(\alpha)\alpha$} and hence defines an idempotent
$\tilde{\pi}_{\chi}$ in $\mathrm{End}_{\mathcal{X}}(A)$ given by
the tuple {\small $\left(\frac{\chi(\alpha)}{|G|}\alpha\right)_{\alpha\in G}$}.

For an arbitrary subgroup $H$ of $G$, we have a natural surjection $\hat{G}\to\hat{H}$ given
by restriction. For $\zeta\in \hat{G}$ and $\chi\in \hat{H}$, we define
$\chi\zeta\in \hat{H}$ via
$\chi\zeta(\alpha):=\chi(\alpha)\zeta(\alpha)$, for $\alpha\in H$.

\begin{lemma}\label{lem1}
{\hspace{1mm}}

\begin{enumerate}[$($i$)$]
\item\label{lem1.1}
Let $M$ be an indecomposable  $A$-$A$-bimodule. Then there is an isomorphism of algebras
\begin{displaymath}
\mathrm{End}_{\mathcal{X}}(M)/\mathrm{Rad}(\mathrm{End}_{\mathcal{X}}(M))\cong
\Bbbk[{G}_M]/\mathrm{Rad}(\Bbbk[G_M])\cong\Bbbk[{G}_M].
\end{displaymath}
\item\label{lem1.2} Indecomposable objects of $\tilde{{\mathcal{X}}}$ are of the form
$M_{\varepsilon_{\chi}}:=(M,\varepsilon_{\chi})$, where
$M$ is an indecomposable  $A$-$A$-bimodule and $\chi\in \hat{G}_M$.
Here, for $\alpha\in G$, the $\alpha$-component of $\varepsilon_{\chi}$ is
$\frac{\chi(\alpha)}{|G_M|}\alpha$, if $\alpha\in G_M$, and zero otherwise.
\end{enumerate}
\end{lemma}

\begin{proof}
Note that, for $\alpha\in G$, if ${}^{\alpha}M^{\alpha}$ is not isomorphic to $M$, then
the $\alpha$-component of any endomorphism of $M$ belongs to the radical of $\mathrm{End}_{\mathcal{X}}(M)$.
Therefore Claim~\eqref{lem1.1} follows from \eqref{eq1}. Claim~\eqref{lem1.2} follows from \eqref{eq1}
and the definitions.
\end{proof}

The category of all $A$-$A$-bimodules has a natural monoidal structure given by the tensor product over $A$.
We define a tensor product on $\mathcal{X}$ by
\begin{itemize}
\item $\displaystyle M\otimes_{\mathcal{X}}N:=\bigoplus_{\varphi\in G}
\big(M\otimes_A {}^{\varphi}N^{\varphi^{}}\big)$, for any $A$-$A$-bimodules $M$ and $N$,
\item $f\otimes_{\mathcal{X}}g:=\big(f_{\alpha}\otimes {}^{\gamma}(g_{\beta\gamma^{\mone}})^{\gamma}\big)_{\alpha,\beta,\gamma\in G}$,
where
\begin{displaymath}
f_{\alpha}\otimes {}^{\gamma}(g_{\beta\gamma^{\mone}})^{\gamma}:
M\otimes_A {}^{\gamma}N^{\gamma^{}}\to {}^{\alpha}(M')^{\alpha}
\otimes_A {}^{\beta}(N')^{\beta},
\end{displaymath}
for any $A$-$A$-bimodules $M$, $M'$, $N$ and $N'$ and morphisms
\begin{displaymath}
f=(f_{\alpha})_{\alpha\in G}\in\mathrm{Hom}_{\mathcal{X}}(M,M'),\quad
g=(g_{\beta})_{\beta\in G}\in\mathrm{Hom}_{\mathcal{X}}(N,N').
\end{displaymath}
\end{itemize}
The asymmetry of the above definition is only notational as the following lemma shows.

\begin{lemma}\label{lem3}
In the category $\mathcal{X}$, there is an isomorphism
\begin{displaymath}
\bigoplus_{\varphi\in G}
\big(M\otimes_A {}^{\varphi}N^{\varphi^{}}\big)\cong \bigoplus_{\varphi\in G}
\big({}^{\varphi^{\mone}}M^{\varphi^{\mone}}\otimes_A N\big).
\end{displaymath}
\end{lemma}

\begin{proof}
We first note that the map $m\otimes n\mapsto m\otimes n$  gives rise to an isomorphism
of $A$-$A$-bimodules from $M\otimes_A{}^{\varphi}N$ to $M^{\varphi^{\mone}}\otimes_A N$. Thus we have an isomorphism
\begin{equation}\label{eq2}
{}^{\varphi^{\mone}}M^{\varphi^{\mone}}\otimes_A N\cong
{}^{\varphi^{\mone}}\left(M\otimes_A{}^{\varphi}N^{\varphi}\right)^{\varphi^{\mone}}
\end{equation}
of $A$-$A$-bimodules. We hence have an isomorphism
\begin{displaymath}
\xymatrix{{}^{\varphi^{\mone}}M^{\varphi^{\mone}}\otimes_A N\ar[rr]^{(f_{\psi})_{\psi\in G}}&&
M\otimes_A{}^{\varphi}N^{\varphi}}
\end{displaymath}
where $f_{\varphi^{\mone}}$ is given by \eqref{eq2} and the remaining components are zero.
\end{proof}

\begin{remark}\label{rem3.5}
{\em
Under the isomorphism provided by Lemma~\ref{lem3}, the morphism $f\otimes_{\mathcal{X}}g$
in $\mathrm{Hom}_{\mathcal{X}}(M\otimes_{\mathcal{X}}N,M'\otimes_{\mathcal{X}}N')$ has components of the form
\begin{displaymath}
{}^{\gamma}(f_{\alpha\gamma^{\mone}})^{\gamma}\otimes g_{\beta}:
{}^{\gamma}M^{\gamma^{}}\otimes_A N\to {}^{\alpha}(M')^{\alpha}\otimes_A {}^{\beta}(N')^{\beta}.
\end{displaymath}
}
\end{remark}

\begin{lemma}\label{lem2}
{\hspace{1mm}}

\begin{enumerate}[$($i$)$]
\item\label{lem2.1}
The operation $\otimes_{\mathcal{X}}$ is bifunctorial.
\item\label{lem2.2}
If $e$ and $f$ are idempotents in $\mathcal{X}$, then so is $e\otimes_{\mathcal{X}} f$.
Hence $\otimes_{\mathcal{X}}$ extends to a bifunctor
$\otimes_{\tilde{\mathcal{X}}}:\tilde{\mathcal{X}}\times \tilde{\mathcal{X}}\to \tilde{\mathcal{X}}$
given by $(M,e)\otimes_{\tilde{\mathcal{X}}}(N,f)=(M\otimes_{\mathcal{X}}N,e\otimes_{\mathcal{X}} f)$.
\end{enumerate}
\end{lemma}

\begin{proof}
Let $M$, $N$, $K$, $L$, $X$ and $Y$ be objects in $\mathcal{X}$. Let  $f:M\to K$, $g:N\to L$, $h:K\to X$
and $l:L\to Y$ be morphisms in $\mathcal{X}$ with their only non-zero components being $f_{\alpha}$,
$g_{\gamma\varphi^{\mone}}$, $h_{\beta\alpha^{\mone}}$ and $l_{\delta\gamma^{\mone}}$, respectively. Consider the diagram
\begin{displaymath}
\xymatrix{
M\ar[d]_{f_{\alpha}}&\otimes_{A}& {}^{\varphi}N^{\varphi}\ar[d]_{{}^{\varphi}(g_{\gamma\varphi^{\mone}})^{\varphi}}\\
{}^{\alpha}K^{\alpha}\ar[d]_{{}^{\alpha}(h_{\beta\alpha^{\mone}})^{\alpha}}&\otimes_{A}&
{}^{\gamma}L^{\gamma}\ar[d]_{{}^{\gamma}(l_{\delta\gamma^{\mone}})^{\gamma}}\\
{}^{\beta}X^{\beta}&\otimes_{A}& {}^{\delta}Y^{\delta}.
}
\end{displaymath}
Bifunctoriality of $\otimes_{A}$ yields
{\small
\begin{equation}\label{eq3}
\left({}^{\alpha}(h_{\beta\alpha^{\mone}})^{\alpha}\circ f_{\alpha}\right)\otimes
\left({}^{\gamma}(l_{\delta\gamma^{\mone}})^{\gamma}\circ{}^{\varphi}(g_{\gamma\varphi^{\mone}})^{\varphi}\right)=
\left({}^{\alpha}(h_{\beta\alpha^{\mone}})^{\alpha}\otimes{}^{\gamma}(l_{\delta\gamma^{\mone}})^{\gamma}\right)\circ
\left(f_{\alpha}\otimes{}^{\varphi}(g_{\gamma\varphi^{\mone}})^{\varphi}\right).
\end{equation}}
The left hand side and the right hand side of \eqref{eq3} coincide with
\begin{displaymath}
(h\circ f)_{\beta}\otimes {}^{\varphi}(l\circ g)_{\delta\varphi^{\mone}}^{\varphi} \quad\text{ and }\quad
{}^{\alpha}(h\otimes l)^{\alpha}_{\beta\alpha^{\mone}}\circ (f\otimes g)_{\alpha},
\end{displaymath}
respectively. As these are the only non-zero components in $(h\circ f)\otimes_{\mathcal{X}}(l\circ g)$
and $(h\otimes_{\mathcal{X}} l)\circ(f\otimes_{\mathcal{X}} g)$, 
respectively, and have the same source and target, we obtain
\begin{equation}\label{eq4}
(h\circ f)\otimes_{\mathcal{X}}(l\circ g)=(h\otimes_{\mathcal{X}} l)\circ(f\otimes_{\mathcal{X}} g)
\end{equation}
in this case. The general case of equality~\eqref{eq4} follows by linearity, implying Claim~\eqref{lem2.1}.

Claim~\eqref{lem2.2} follows from equality~\eqref{eq4}.
\end{proof}

\begin{proposition}\label{propmultid}
For any $\chi,\zeta\in \hat{G}$, we have
\begin{equation}\label{eq5}
(A,\tilde{\pi}_{\chi})\otimes_{\tilde{\mathcal{X}}}(A,\tilde{\pi}_{\zeta})\cong (A,\tilde{\pi}_{\chi\zeta}).
\end{equation}
\end{proposition}

\begin{proof}
We start by constructing a morphism $\boldsymbol{f}$ from the right hand side of \eqref{eq5} to the left hand side, which is
defined as follows:
\begin{displaymath}
\boldsymbol{f}:=(f_{\sigma,\tau})_{\sigma,\tau\in G}\colon A\to
\displaystyle\bigoplus_{\sigma,\tau\in G}{}^{\sigma}A^{\sigma}\otimes_A {}^{\tau}A^{\tau}
\end{displaymath}
where $f_{\sigma,\tau}$ is given by
\begin{equation}\label{eq6}
1 \mapsto \frac{1}{|G|^2}\chi(\sigma)\zeta(\tau)(1\otimes 1)\in {}^{\sigma}A^{\sigma}\otimes_A {}^{\tau}A^{\tau}.
\end{equation}

Consider the diagram
\begin{displaymath}
\xymatrix{
A\ar[rrr]^{\tilde{\pi}_{\chi\zeta}}\ar[d]_{(f_{\sigma,\tau})_{\sigma,\tau\in G}}
&&&\displaystyle\bigoplus_{\beta\in G}{}^{\beta}A^{\beta}
\ar[d]^{\big({}^{\beta}(f_{\alpha\beta^{\mone},\delta\beta^{\mone}})_{\alpha\beta^{\mone},\delta\beta^{\mone}\in G}^{\beta}\big)_{\beta\in G}}\\
\displaystyle\bigoplus_{\sigma,\tau\in G}{}^{\sigma}A^{\sigma}\otimes_A {}^{\tau}A^{\tau}
\ar[rrr]^{({}^{\sigma}(\tilde{\pi}_{\chi}\otimes_{\mathcal{X}}\tilde{\pi}_{\zeta})^{\sigma})_{\sigma\in G}}
&&&\displaystyle
\bigoplus_{\alpha,\delta\in G}
{}^{\alpha}A^{\alpha}\otimes_A {}^{\delta}A^{\delta}.
}
\end{displaymath}
By definition,  the $\tau\sigma^{\mone},\alpha\sigma^{\mone},\delta\sigma^{\mone}$-component of
$\tilde{\pi}_{\chi}\otimes_{\mathcal{X}}\tilde{\pi}_{\zeta}$ is multiplication by the scalar
\begin{equation}\label{eq7}
\frac{1}{|G|^2} \chi(\alpha\sigma^{\mone})\zeta(\delta\tau^{\mone})=
\frac{1}{|G|^2} \chi(\alpha)\chi(\sigma^{\mone})\zeta(\delta)\zeta(\tau^{\mone}).
\end{equation}

Now we compute the components of the two compositions $(\tilde{\pi}_{\chi}\otimes_{\mathcal{X}}\tilde{\pi}_{\zeta})\circ \boldsymbol{f}$
(corresponding to the path going down and then right) and
$\boldsymbol{f}\circ \tilde{\pi}_{\chi\zeta}$ (corresponding to the path going right and then down)
in our diagram which end up in a specific ${}^{\alpha}A^{\alpha}\otimes_A {}^{\delta}A^{\delta}$.
One way around,
using \eqref{eq6} and \eqref{eq7}, we obtain that $1$ is sent to
\begin{displaymath}
\begin{array}{rcl}
\frac{1}{|G|^4}\displaystyle\sum_{\sigma,\tau\in G} \chi(\sigma)\chi(\alpha)\chi(\sigma^{\mone})\zeta(\tau)\zeta(\delta)\zeta(\tau^{\mone}) (1\otimes 1)
&=&\frac{|G|^2}{|G|^4}\chi(\alpha)\zeta(\delta) (1\otimes 1)\\
&=&\frac{1}{|G|^2}\chi(\alpha)\zeta(\delta) (1\otimes 1).
\end{array}
\end{displaymath}
The other way around, using \eqref{eq6} we obtain that $1$ is sent to
\begin{displaymath}
\begin{array}{rcl}
\frac{1}{|G|^3}\displaystyle\sum_{\beta\in G} \chi\zeta(\beta)\chi(\alpha\beta^{\mone})\zeta(\delta\beta^{\mone})(1\otimes 1)
&=&\frac{1}{|G|^2}\chi(\alpha)\zeta(\delta) (1\otimes 1).
\end{array}
\end{displaymath}
Hence the diagram commutes and, moreover, \eqref{eq00} is satisfied.
Thus $\boldsymbol{f}$ represents a morphism from the right hand side of \eqref{eq5} to the left hand side.

We proceed by constructing a morphism $\boldsymbol{g}$ from the left hand side of \eqref{eq5} to the right hand side.
Consider the diagram
\begin{displaymath}
\xymatrix{
\displaystyle\bigoplus_{\alpha\in G}(A\otimes_A {}^{\alpha}A^{\alpha})
\ar[rrr]^{\tilde{\pi}_{\chi}\otimes_{\mathcal{X}}\tilde{\pi}_{\zeta}}
\ar[d]|-{\frac{1}{|G|}\big(\chi(\beta)\zeta(\beta\alpha^{\mone})\big)_{\alpha,\beta\in G}}
&&&
\displaystyle\bigoplus_{\gamma,\delta\in G}{}^{\gamma}A^{\gamma}\otimes_A{}^{\delta}A^{\delta}
\ar[d]|-{\frac{1}{|G|}\big({}^{\gamma}(\chi(\sigma\gamma^{\mone})\zeta(\sigma\delta^{\mone}))_{\delta\gamma^{\mone},\sigma\gamma^{\mone}\in G}^{\gamma}\big)_{\gamma\in G}}\\
\displaystyle\bigoplus_{\beta\in G} {}^{\beta}A^{\beta}
\ar[rrr]^{({}^{\beta}\tilde{\pi}_{\chi\zeta}^{\beta})_{\beta\in G}}
&&&
\displaystyle\bigoplus_{\sigma\in G}{}^{\sigma}A^{\sigma}
}
\end{displaymath}
whose vertical part defines $\boldsymbol{g}$.
For fixed $\alpha,\sigma\in G$, going one way around, using \eqref{eq7} we obtain
\begin{displaymath}
\sum_{\gamma,\delta\in G}\frac{1}{|G|^3} \chi(\gamma)\zeta(\delta\alpha^{\mone})\chi(\sigma\gamma^{\mone})\zeta(\sigma\delta^{\mone})=\frac{1}{|G|}
\chi(\sigma)\zeta(\sigma\alpha^{\mone}).
\end{displaymath}
The other way around yields
\begin{displaymath}
\frac{1}{|G|^2}\sum_{\beta\in G}\chi(\beta)\zeta(\beta\alpha^{\mone})\chi\zeta(\sigma\beta^{\mone})=
\frac{1}{|G|}\chi(\sigma)\zeta(\sigma\alpha^{\mone})
\end{displaymath}
which implies that the diagram commutes and, moreover, $\boldsymbol{g}\circ (\tilde{\pi}_{\chi}\otimes_{\mathcal{X}}\tilde{\pi}_{\zeta})=\tilde{\pi}_{\chi\zeta}\circ \boldsymbol{g}=\boldsymbol{g}$.

Now we claim that both compositions $\boldsymbol{f}\circ\boldsymbol{g}$ and $\boldsymbol{g}\circ\boldsymbol{f}$
are the identities, i.e. of the respective idempotents. The $\varphi$-component of the
composition $\boldsymbol{g}\circ\boldsymbol{f}$ sends $1$ to
\begin{displaymath}
\frac{1}{|G|^3}\sum_{\sigma,\tau\in G}\chi(\sigma)\zeta(\tau)\chi(\varphi\sigma^{\mone})\zeta(\varphi\tau^{\mone})=
\frac{1}{|G|}\chi\zeta(\varphi).
\end{displaymath}
The $\alpha,\sigma,\tau$-component of the composition $\boldsymbol{f}\circ\boldsymbol{g}$ sends $1\otimes 1$ to
\begin{displaymath}
\frac{1}{|G|^3}\sum_{\beta\in G}\chi(\beta)\zeta(\beta\alpha^{\mone})\chi(\sigma\beta^{\mone})\zeta(\tau\beta^{\mone})(1\otimes 1)=
\frac{1}{|G|^2}\chi(\sigma)\zeta(\tau\alpha^{\mone})(1\otimes 1).
\end{displaymath}
The claim follows.
\end{proof}

\begin{proposition}\label{propmultid2}
Let $i,j\in\{1,2,\dots,k\}$ and $M=Ae_i\otimes_{\Bbbk}e_jA$.
Let further $\chi\in \hat{G}_M$ and $\zeta\in \hat{G}$. Then
\begin{equation}\label{eq8}
(M,\varepsilon_{\chi})\otimes_{\tilde{\mathcal{X}}}(A,\tilde{\pi}_{\zeta})\cong (M,\varepsilon_{\chi\zeta}).
\end{equation}
\end{proposition}

\begin{proof}
We follow the proof of Proposition~\ref{propmultid}.
We start by constructing a morphism $\boldsymbol{f}$ from the right hand side of \eqref{eq8} to the left hand side.
Consider the morphism
\begin{displaymath}
\boldsymbol{f}:=(f_{\sigma,\tau})_{\sigma,\tau\in G}\colon M\to
\displaystyle\bigoplus_{\sigma,\tau\in G}{}^{\sigma}M^{\sigma}\otimes_A {}^{\tau}A^{\tau}
\end{displaymath}
where $f_{\sigma,\tau}$ is given by
\begin{equation*}
e_i\otimes e_j \mapsto \frac{1}{|G_M||G|}\chi(\sigma)\zeta(\tau)(\sigma(e_i)\otimes \sigma(e_j)\otimes 1)\in
{}^{\sigma}M^{\sigma}\otimes_A {}^{\tau}A^{\tau}.
\end{equation*}
if $\sigma\in G_M$ and zero otherwise.
Consider the diagram
\begin{displaymath}
\xymatrix{
M\ar[rrr]^{\varepsilon_{\chi\zeta}}\ar[d]_{(f_{\sigma,\tau})_{\sigma,\tau\in G}}
&&&\displaystyle\bigoplus_{\beta\in G}{}^{\beta}M^{\beta}
\ar[d]^{({}^{\beta}(f_{\alpha\beta^{\mone},\delta\beta^{\mone}})_{\alpha\beta^{\mone},\delta\beta^{\mone}\in G}^{\beta})_{\beta\in G}}\\
\displaystyle\bigoplus_{\sigma,\tau\in G}{}^{\sigma}M^{\sigma}\otimes_A {}^{\tau}A^{\tau}
\ar[rrr]^{({}^{\sigma}({\varepsilon}_{\chi}\otimes_{\mathcal{X}}\tilde{\pi}_{\zeta})^{\sigma})_{\sigma\in G}}
&&&\displaystyle
\bigoplus_{\alpha,\delta\in G}
{}^{\alpha}M^{\alpha}\otimes_A {}^{\delta}A^{\delta}.
}
\end{displaymath}
By definition,  the $\tau\sigma^{\mone},\alpha\sigma^{\mone},\delta\sigma^{\mone}$-component of
$\varepsilon_{\chi}\otimes_{\mathcal{X}}\tilde{\pi}_{\zeta}$
sends $e_i\otimes e_j\otimes 1$ to
\begin{displaymath}
\frac{1}{|G_M||G|}\chi(\alpha)\chi(\sigma^{\mone})\zeta(\delta)\zeta(\tau^{\mone})
(\alpha\sigma^{\mone}(e_i)\otimes\alpha\sigma^{\mone}(e_j)\otimes 1)
\end{displaymath}
if $\alpha\sigma^{\mone}\in G_M$ and zero otherwise.
Now, going to the right and then down, the $\alpha,\delta$-component of the composition
$\boldsymbol{f}\circ \varepsilon_{\chi\zeta}$
sends $e_i\otimes e_j$ to
\begin{displaymath}
\frac{1}{|G_M||G|}\chi(\alpha)\zeta(\delta)(\alpha(e_i)\otimes \alpha(e_j)\otimes 1),
\end{displaymath}
if $\alpha$ is in $G_M$, and zero otherwise. Going down and then to the right,
the $\alpha,\delta$-component of
$(\varepsilon_{\chi}\otimes_{\mathcal{X}}\tilde{\pi}_{\zeta})\circ \boldsymbol{f}$
gives the same result, which also equals $f_{\alpha,\delta}$.

To construct a morphism $\boldsymbol{g}$ from the left hand side of \eqref{eq8} to the right hand side,
consider the diagram
\begin{displaymath}
\xymatrix{\displaystyle\bigoplus_{\alpha\in G}(M\otimes_A {}^{\alpha}A^{\alpha})
\ar[rrr]^{\varepsilon_{\chi}\otimes_{\mathcal{X}}\tilde{\pi}_{\zeta}}
\ar[d]_{(g_{\beta,\alpha})_{\alpha,\beta\in G}}
&&&
\displaystyle\bigoplus_{\gamma,\delta\in G}{}^{\gamma}M^{\gamma}\otimes_A{}^{\delta}A^{\delta}
\ar[d]^{\big({}^{\gamma}(g_{\sigma\gamma^{\mone},\delta\gamma^{\mone}})_{\delta\gamma^{\mone},\sigma\gamma^{\mone}\in G}^{\gamma}\big)_{\delta\in G}}
\\
\displaystyle\bigoplus_{\beta\in G} {}^{\beta}M^{\beta}
\ar[rrr]^{({}^{\beta}\varepsilon_{\chi\zeta}^{\beta})_{\beta\in G}}
&&&
\displaystyle\bigoplus_{\sigma\in G}{}^{\sigma}M^{\sigma}.
}
\end{displaymath}
where $g_{\beta,\alpha}$ sends $e_i\otimes e_j\otimes 1\in M\otimes_A {}^{\alpha}A^{\alpha}$
to $\frac{1}{|G_M|}\chi(\beta)\zeta(\beta\alpha^{\mone})(\beta(e_i)\otimes\beta(e_j))\in {}^{\beta}M^{\beta}$,
if $\beta\in G_M$, and to zero otherwise.

For fixed $\alpha,\sigma\in G$, going one way around we obtain the map which sends
$e_i\otimes e_j\otimes 1$ to
\begin{displaymath}
\frac{1}{|G_M|}
\chi(\sigma)\zeta(\sigma\alpha^{\mone})(\sigma(e_i)\otimes
\sigma(e_j)),
\end{displaymath}
if $\sigma\in G_M$, and to zero otherwise, which coincides with $g_{\sigma,\alpha}$.
The other way around gives the same result.

Checking that both compositions $\boldsymbol{f}\circ\boldsymbol{g}$ and $\boldsymbol{g}\circ\boldsymbol{f}$
are the respective idempotents is similar to the proof of Proposition~\ref{propmultid}.
\end{proof}

\subsection{Tensoring symmetric bimodules with $A$-modules}\label{s3.2}

\begin{proposition}\label{prop4}
{\hspace{1mm}}

\begin{enumerate}[$($i$)$]
\item\label{prop4.1}
There is a bifunctor $\otimes^{(r)}:\mathrm{mod}\text{-}A\times \mathcal{X}\to \mathrm{mod}\text{-}A$
defined by
\begin{displaymath}
\begin{array}{rcl}
(V,M)&\mapsto& \displaystyle V\otimes^{(r)}M:= \bigoplus_{\varphi\in G}V\otimes_A {}^{\varphi}M^{\varphi}\\
(f,g)&\mapsto& f\otimes^{(r)} g:=\big(f\otimes {}^{\varphi}(g_{\alpha\varphi^{\mone}})^{\varphi}\big)_{\varphi,\alpha\in G}.
\end{array}
\end{displaymath}
\item\label{prop4.2}
There is a bifunctor $\otimes^{(l)}:\mathcal{X}\times A\text{-}\mathrm{mod}\to A\text{-}\mathrm{mod}$
defined by
\begin{displaymath}
\begin{array}{rcl}
(M,V)&\mapsto& \displaystyle M\otimes^{(l)}V:= \bigoplus_{\varphi\in G}{}^{\varphi}M^{\varphi}\otimes_A V\\
(g,f)&\mapsto& g\otimes^{(l)} f:=\big({}^{\varphi}(g_{\alpha\varphi^{\mone}})^{\varphi}\otimes f\big)_{\alpha,\varphi\in G}.
\end{array}
\end{displaymath}
\end{enumerate}
\end{proposition}

\begin{proof}
We note that we use Lemma~\ref{lem3} for the formulation of Claim~\eqref{prop4.2}. The proof of
both claims is similar to the proof of Lemma~\ref{lem2}\eqref{lem2.1}.
\end{proof}

\begin{proposition}\label{prop5}
{\hspace{1mm}}

\begin{enumerate}[$($i$)$]
\item\label{prop5.1}
The bifunctor $\otimes^{(r)}$ induces
a bifunctor $\mathrm{mod}\text{-}A\times \tilde{\mathcal{X}}\to \mathrm{mod}\text{-}A$
(which we will denote by the same symbol abusing notation).
\item\label{prop5.2}
The  bifunctor $\otimes^{(l)}$ induces
a bifunctor $\tilde{\mathcal{X}}\times A\text{-}\mathrm{mod}\to A\text{-}\mathrm{mod}$
(which we will denote by the same symbol abusing notation).
\end{enumerate}
\end{proposition}

\begin{proof}
Let  $(M,e)\in \tilde{\mathcal{X}}$. Then, for any $V\in \mathrm{mod}\text{-}A$,
the endomorphism $\mathrm{id}_V \otimes^{(r)} e$ is an idempotent endomorphism of
$V\otimes^{(r)}M$, so we can define $V\otimes^{(r)}(M,e)$ as the image of
this idempotent. It is easy to check that this does the job for Claim~\eqref{prop5.1}.
Claim~\eqref{prop5.2} is similar.
\end{proof}

\subsection{The $2$-category $\cG_{A}$ of projective symmetric bimodules}\label{s3.3}

Assume that we are in the setup of Subsection~\ref{s3.1}. Let
\begin{displaymath}
A=A_1\times A_2\times \dots \times A_n
\end{displaymath}
be the (unique up to permutation of factors) decomposition of $A$ into a direct product of
indecomposable algebras. Assume that the action of each $\varphi\in G$ preserves each $A_i$.
Also assume that none of $A_i$ is simple. We also consider the algebra $B:=A\times \Bbbk$.

For each $i\in\{1,2,\dots,n\}$, fix a small category $\mathcal{C}_i$ equivalent to $A_i\text{-}\mathrm{proj}$.
Define the $2$-category $\cG_{A}$ to have
\begin{itemize}
\item objects $\mathtt{1},\mathtt{2},\dots,\mathtt{n}$, where we identify $\mathtt{i}$ with
$\mathcal{C}_i$;
\item $1$-morphisms are endofunctors of $\mathcal{C}:=\coprod_{\mathtt{i}}\mathcal{C}_i$ isomorphic to
functors $X\otimes^{(l)}{}_-$, where $X$ is in the additive closure of
$\left(A\oplus(A\otimes_{\Bbbk}A),\mathrm{id}_{A\oplus(A\otimes_{\Bbbk}A)}\right)$
inside $\tilde{\mathcal{X}}$;
\item $2$-morphisms are given by morphisms between $X$ and $X'$ in $\tilde{\mathcal{X}}$;
\item horizontal composition is just composition of functors;
\item vertical composition is inherited from $\tilde{\mathcal{X}}$;
\item the identity $1$-morphism in $\cG_{A}(\mathtt{i},\mathtt{i})$ is isomorphic to
$(A_i,\tilde{\pi}_{1_{\hat{G}}})\otimes^{(l)}{}_-$.
\end{itemize}
Note that the restriction on $\mathrm{char}(\Bbbk)$ as not dividing the order of $G$
is necessary to have identity $1$-morphisms.
Observe further that  $A\oplus(A\otimes_{\Bbbk}A)$ is invariant, up to isomorphism, under the functor
$M\mapsto {}^{\varphi}M^{\varphi}$, for any $\varphi\in G$. The fact that this defines a
$2$-category is justified by Proposition~\ref{propmultid2}, showing that
$(A_i,\tilde{\pi}_{1_{\hat{G}}})\otimes^{(l)}{}_-$ is indeed an identity, and the following lemma.

\begin{lemma}\label{lem6}
Let $X$ and $Y$ be  in the additive closure of
$\left(A\oplus(A\otimes_{\Bbbk}A),\mathrm{id}_{A\oplus(A\otimes_{\Bbbk}A)}\right)$
inside $\tilde{\mathcal{X}}$. Then there is an isomorphism
\begin{displaymath}
(X\otimes^{(l)}{}_-)\circ(Y\otimes^{(l)}{}_-)\cong
(X\otimes_{\tilde{\mathcal{X}}} Y)\otimes^{(l)}{}_-
\end{displaymath}
of endofunctors of $\mathcal{C}$.
\end{lemma}

\begin{proof}
First we assume that $X$ and $Y$ are in $\mathcal{X}$. Then, for any $P\in \mathcal{C}$, we have
\begin{displaymath}
(X\otimes_{\mathcal{X}}Y)\otimes^{(l)}P=
\left(\bigoplus_{\varphi\in G}X\otimes_A {}^{\varphi} Y^{\varphi}\right) \otimes^{(l)}P=
\bigoplus_{\varphi,\psi\in G}{}^{\psi}(X\otimes_A {}^{\varphi} Y^{\varphi})^{\psi}\otimes_A P
\end{displaymath}
and
\begin{displaymath}
X\otimes^{(l)}(Y\otimes^{(l)}P)=
X\otimes^{(l)}\left(\bigoplus_{\varphi\in G}{}^{\varphi} Y^{\varphi} \otimes_A P\right)=
\bigoplus_{\varphi,\psi\in G}{}^{\psi}X^{\psi}\otimes_A {}^{\varphi} Y^{\varphi}\otimes_A P.
\end{displaymath}
Choosing an isomorphism
\begin{displaymath}
\bigoplus_{\varphi,\psi\in G}{}^{\psi}(X\otimes_A {}^{\varphi} Y^{\varphi})^{\psi}
\cong \bigoplus_{\varphi,\psi\in G}{}^{\psi}X^{\psi}\otimes_A {}^{\varphi} Y^{\varphi}
\end{displaymath}
of $A$-$A$-bimodules yields the desired isomorphism of functors.

Now, let $e$ and $f$ be idempotents in $\mathrm{End}_{\mathcal{X}}(X)$ and $\mathrm{End}_{\mathcal{X}}(Y)$,
respectively. Consider
\begin{displaymath}
\begin{array}{rcl}
(e\otimes_{\mathcal{X}}f)\otimes^{(l)}\mathrm{id}_P&=&
\left({}^{\gamma}(e_{\alpha\gamma^{\mone}})^{\gamma}\otimes f_{\beta}\right)_{\alpha,\beta,\gamma\in G}\otimes^{(l)}\mathrm{id}_P\\
&=&\left({}^{\delta}\left({}^{\gamma}(e_{\alpha\gamma^{\mone}})^{\gamma}\otimes f_{\beta}
\right)^{\delta}_{\beta\delta^{\mone}}\otimes\mathrm{id}_P\right)_{\alpha,\beta,\gamma,\delta\in G}\\
&=&\left({}^{\gamma\delta}(e_{\alpha\gamma^{\mone}})^{\gamma\delta}\otimes {}^{\delta}(f_{\beta\delta^{\mone}})^{\delta}
\otimes\mathrm{id}_P\right)_{\alpha,\beta,\gamma,\delta\in G}
\end{array}
\end{displaymath}
and
\begin{displaymath}
\begin{array}{rcl}
e\otimes^{(l)}(f\otimes^{(l)}\mathrm{id}_P)&=&
e\otimes^{(l)}\left({}^{\psi}(f_{\varphi\psi^{\mone}})^{\psi}\otimes \mathrm{id}_P\right)_{\varphi,\psi\in G}\\
&=&
\left({}^{\tau}(e_{\sigma\tau^{\mone}})^{\tau}\otimes
{}^{\psi}(f_{\varphi\psi^{\mone}})^{\psi}\otimes \mathrm{id}_P\right)_{\sigma,\tau,\varphi,\psi\in G},
\end{array}
\end{displaymath}
we see that
\begin{displaymath}
(e\otimes_{\mathcal{X}}f)\otimes^{(l)}\mathrm{id}_P=e\otimes^{(l)}(f\otimes^{(l)}\mathrm{id}_P)
\end{displaymath}
and hence the isomorphism in the previous paragraph descends to the summands $(X,e)$ and $(Y,f)$.
\end{proof}

The $2$-category $\cG_B$ is defined similarly. To distinguish the underlying categories of bimodules,
we use the notation $\mathcal{Y}$ and $\tilde{\mathcal{Y}}$ for the corresponding
categories of symmetric $B$-$B$-bimodules.

\subsection{Two-sided cells in $\cG_{A}$}\label{s3.4}

We recall the notation introduced just before Lemma~\ref{lem1}.

\begin{proposition}\label{prop7}
The $2$-category $\cG_{A}$ has $n+1$ two-sided cells, namely

\begin{enumerate}[$($a$)$]
\item\label{prop7.1} for $i=1,2,\dots,n$, the two-sided cell $\mathcal{J}_{\mathtt{i}}$
consisting of $|G|$ elements $(\mathbbm{1}_{\mathtt{i}},\tilde{\pi}_{\varphi})$, where $\varphi\in G$,
\item\label{prop7.2}
the two-sided cell $\mathcal{J}_{0}$ consisting of all isomorphism classes of indecomposable
$1$-morphisms in the  additive closure of
$\left(A\otimes_{\Bbbk}A,\mathrm{id}_{A\otimes_{\Bbbk}A}\right)$
inside $\tilde{\mathcal{X}}$.
\end{enumerate}
\end{proposition}

\begin{proof}
Since tensor products in which one of the factors is a  projective bimodule never contain a copy of
the regular bimodule as a direct summand, the existence of two-sided cells as claimed in
Part~\eqref{prop7.1} follows from Proposition~\ref{propmultid}.
To complete the proof of the proposition, it remains to show that all isomorphism classes of indecomposable
$1$-morphisms in the  additive closure of
$\left(A\otimes_{\Bbbk}A,\mathrm{id}_{A\otimes_{\Bbbk}A}\right)$
inside $\tilde{\mathcal{X}}$ belong to the same two-sided cell.
Ignoring idempotents in $\mathcal{X}$, the claim follows directly from \cite[Subsection~5.1]{MM5}.
In full generality, the statement is then proved using Proposition~\ref{propmultid2}.
\end{proof}

\subsection{Adjunctions}\label{s3.6}

In this subsection, we study adjunctions in the $2$-category $\cG_A$ under the assumption that $A$
is self-injective. We assume that $A$ is basic and that there is a fixed complete $G$-invariant
set $\mathtt{E}$ of primitive idempotents. We denote by $\nu$ the bijection on $\mathtt{E}$
which is induced by the Nakayama automorphism of $A$ given by
\begin{displaymath}
\mathrm{Hom}_{\Bbbk}(eA,\Bbbk)\cong A\nu(e), \quad\text{ for } e\in  \mathtt{E}.
\end{displaymath}
For a primitive idempotent $e\in A$, we denote by $\varepsilon_e$ the idempotent
in $\mathrm{End}_{\mathcal{Y}}(Ae)$ or $\mathrm{End}_{\mathcal{Y}}(eA)$ corresponding to the
trivial character of $G_{A\nu(e)}=G_{Ae}=G_{eA}=:G_e$. We denote by $\mathbf{m}$ the multiplication map in $B$.

\begin{proposition}\label{prop61}
We have adjunctions
\begin{enumerate}[$($a$)$]
\item\label{prop61.1}
$((Ae,\varepsilon_{e}),(eA,\varepsilon_{e}))$ in $\tilde{\mathcal{Y}}$;
\item\label{prop61.2}
$((eA,\varepsilon_{e}),(A\nu(e),\varepsilon_{\nu(e)}))$ in $\tilde{\mathcal{Y}}$.
\end{enumerate}
\end{proposition}

\begin{proof}
We first define the counit
\begin{displaymath}
\epsilon:(Ae\otimes_{\mathcal{Y}}eA,\varepsilon_{e}\otimes_{\mathcal{Y}}\varepsilon_{e})\to
(A,\tilde{\pi}_{1_{\hat{G}}}).
\end{displaymath}
This is defined by the vertical part of the diagram
\begin{displaymath}
\xymatrix{
\displaystyle \bigoplus_{\varphi}Ae\otimes_{\Bbbk} {}^{\varphi}eA^{\varphi}
\ar[rrrr]^{\frac{1}{|G_{e}|^2}(\alpha\otimes\beta\varphi^{\mone})_{\alpha,\beta\varphi^{\mone}\in G_e}}
\ar[d]^{\frac{1}{|G_{e}|}(\gamma\circ\mathbf{m})_{\gamma,\varphi\in G}}
&&&&\displaystyle \bigoplus_{\alpha,\beta}{}^{\alpha} Ae^{\alpha}\otimes_{\Bbbk} {}^{\beta}eA^{\beta}
\ar[d]^{\frac{1}{|G_{e}|}(\delta\circ\mathbf{m}\circ(\alpha\otimes\alpha)^{\mone})_{\delta,\alpha,\beta\in G}}
\\
\displaystyle \bigoplus_{\gamma}{}^{\gamma}A^{\gamma}
\ar[rrrr]^{\frac{1}{|G|}(\delta\gamma^{\mone})_{\delta,\gamma\in G}}
&&&& \displaystyle \bigoplus_{\delta}{}^{\delta}A^{\delta}
}
\end{displaymath}
where here and in the rest of the proof all elements indexing direct sums run through $G$,
and the notation $(\alpha\otimes\beta\varphi^{\mone})_{\alpha,\beta\varphi^{\mone}\in G_e}$  should be read
as the $(\alpha,\beta,\varphi)$-component of the map being defined as zero if the conditions
$\alpha,\beta\varphi^{\mone}\in G_e$ are not satisfied.

To check that the vertical arrows define a morphism
$(Ae\otimes_{\mathcal{Y}}eA,\varepsilon_{e}\otimes_{\mathcal{Y}}\varepsilon_{e})\to
(A,\tilde{\pi}_{1_{\hat{G}}})$, we need to verify that the diagram commutes and the result coincides with
the original map, namely, satisfying ~\eqref{eq00}.
We consider the $(\delta,\varphi)$-component of both compositions. First going to the right and then down,
$e\otimes\varphi(e)$ is mapped to
\begin{displaymath}
\frac{1}{|G_e|^3}\sum_{\alpha\in G_e,\beta\in \varphi G_e}\delta(e\alpha^{\mone}\beta(e))=
\begin{cases}
\frac{1}{|G_e|}\delta(e), & \text{ if } \varphi\in G_e,\\
0, & \text{ otherwise}.
\end{cases}
\end{displaymath}
The other way around,  $e\otimes\varphi(e)$ is sent to
\begin{displaymath}
\frac{1}{|G_e||G|}\sum_{\gamma\in G}\delta(e\varphi(e))=
 \begin{cases}
\frac{1}{|G_e|}\delta(e), & \text{ if } \varphi\in G_e,\\
0, & \text{ otherwise},
\end{cases}
\end{displaymath}
which coincides with the image of $\frac{1}{|G_{e}|}(\delta\circ\mathbf{m})_{\delta,\varphi\in G}$ on $e\otimes\varphi(e)$.
So
our counit $\epsilon$ is, indeed, well-defined.

We now define the unit
\begin{displaymath}
\eta:(\Bbbk,\tilde{\pi}_{1_{\hat{G}}})\to
(eA\otimes_{\mathcal{Y}}Ae,\varepsilon_{e}\otimes_{\mathcal{Y}}\varepsilon_{e})
\end{displaymath}
by the vertical part of the diagram
\begin{displaymath}
\xymatrix{
\Bbbk\ar[rrrr]^{\frac{1}{|G|}(\varphi)_{\varphi\in G}}
\ar[d]_{(\eta_{1}^{\delta,\alpha})_{\delta\alpha^{\mone}\in G_e}}
&&&&\displaystyle \bigoplus_{\varphi}{}^{\varphi}\Bbbk^{\varphi}
\ar[d]^{(\eta_{\varphi}^{\beta,\gamma})_{\beta\gamma^{\mone}\in G_e,\varphi\in G}}\\
\displaystyle \bigoplus_{\delta,\alpha}{}^{\delta}eA^{\delta}\otimes_A {}^{\alpha}Ae^{\alpha}
\ar[rrrr]^{\frac{1}{|G_e|^2}(\beta\delta^{\mone}\otimes\gamma\alpha^{\mone})_{\beta\delta^{\mone},\gamma\alpha^{\mone}\in G_e}}
&&&&\displaystyle \bigoplus_{\beta,\gamma}{}^{\beta}eA^{\beta}\otimes_A {}^{\gamma}Ae^{\gamma}
}
\end{displaymath}
where
\begin{equation}\label{eq9}
\eta_{\varphi}^{\beta,\gamma}(\varphi(1))=
\begin{cases}
\frac{1}{|G_e|}\beta(e)\otimes\gamma(e),& \text{ if }\beta\gamma^{\mone}\in G_e,\\
0, & \text{ otherwise}.
\end{cases}
\end{equation}

We again check that this defines a morphism
$(\Bbbk,\tilde{\pi}_{1_{\hat{G}}})\to
(eA\otimes_{\mathcal{Y}}Ae,\varepsilon_{e}\otimes_{\mathcal{Y}}\varepsilon_{e})$ by
verifying~\eqref{eq00}. Computing the $(\beta,\gamma)$-component of the path first going to the
right and then down, we see that
\begin{displaymath}
1\mapsto
\begin{cases}
\frac{1}{|G_e|}\beta(e)\otimes\gamma(e), & \text{ if } \beta\gamma^{\mone}\in G_e,\\
0, & \text{ otherwise},
\end{cases}
\end{displaymath}
that is, $(\eta_{1}^{\beta,\gamma})_{\beta\gamma^{\mone}\in G_e}$.
The other way around,
\begin{displaymath}
1\mapsto
\sum_{\delta\in\beta G_e,\alpha\in\gamma G_e}
\frac{1}{|G_e|^3}\beta(e)\otimes\gamma(e)=
\begin{cases}
\frac{1}{|G_e|}\beta(e)\otimes\gamma(e), & \text{ if } \beta\gamma^{\mone}\in G_e,\\
0, & \text{ otherwise}.
\end{cases}
\end{displaymath}
Note that the condition $\delta\alpha^{\mone}\in G_e$ is automatically satisfied for
$\beta\gamma^{\mone}\in G_e$ and $\delta\in\beta G_e,\alpha\in\gamma G_e$.
Thus our unit $\eta$ is well-defined as well.

Now we need to check the adjunction axioms. Denoting $(Ae,\varepsilon_{e})$
by $\mathrm{F}$ and $(eA,\varepsilon_{e})$ by $\mathrm{G}$, we first verify
\begin{displaymath}
\mathrm{F}\to\mathrm{F}\mathbbm{1}_{\mathtt{j}}\to \mathrm{F}\mathrm{G}\mathrm{F}\to
\mathbbm{1}_{\mathtt{i}}\mathrm{F}\to \mathrm{F}
\end{displaymath}
is the identity, for appropriate $\mathtt{i}$ and $\mathtt{j}$.
To this end, consider the commutative diagram
{\small
\begin{displaymath}
\xymatrix{
Ae\ar[rrrr]^{\frac{1}{|G_e|}(\iota)_{\iota\in G_e}}
\ar[d]^{\frac{1}{|G_e||G|}(\kappa\otimes\varphi(1))_{\kappa\in G_e,\varphi\in G}}
&&&&
\displaystyle \bigoplus_{\iota}{}^{\iota}Ae^{\iota}
\ar[d]_{\frac{1}{|G_e||G|}(\alpha\iota^{\mone}\otimes\psi(1))_{\alpha\iota^{\mone}\in G_e,\psi\in G}}
\\
\displaystyle \bigoplus_{\kappa,\varphi}{}^{\kappa}Ae^{\kappa}\otimes_{\Bbbk}{}^{\varphi}\Bbbk^{\varphi}
\ar[rrrr]^{\frac{1}{|G_e||G|}(\alpha\kappa^{\mone}\otimes \psi\varphi^{\mone})_{\alpha\kappa^{\mone}\in G_e,
\psi,\varphi\in G}}
\ar[d]|-{\frac{1}{|G_e|}(\beta\kappa^{\mone}\otimes\eta_{\varphi}^{\gamma,\delta})_{\beta\kappa^{\mone},
\gamma\delta^{\mone}\in G_e,\varphi\in G}}
&&&&
\displaystyle \bigoplus_{\alpha,\psi}{}^{\alpha}Ae^{\alpha}\otimes_{\Bbbk}{}^{\psi}\Bbbk^{\psi}
\ar[d]|-{\frac{1}{|G_e|}(\lambda\alpha^{\mone}\otimes\eta_{\psi}^{\mu,\upsilon})_{\lambda\alpha^{\mone},
\mu\upsilon^{\mone}\in G_e,\psi\in G}}
\\
\displaystyle \bigoplus_{\beta,\gamma,\delta}
{}^{\beta}Ae^{\beta}\otimes_{\Bbbk}{}^{\gamma}eA^{\gamma}\otimes_A{}^{\delta}Ae^{\delta}
\ar[rrrr]^{\frac{1}{|G_e|^3}(\lambda\beta^{\mone}\otimes \mu\gamma^{\mone}\otimes \upsilon\delta^{\mone})}
\ar[d]|-{\frac{1}{|G_e|^2}\big((\xi\beta^{\mone}\circ\mathbf{m})
\otimes\theta\delta^{\mone}\big)_{\theta\delta^{\mone}\in G_e,\xi,\beta,\gamma\in G}}
&&&&
\displaystyle \bigoplus_{\lambda,\mu,\upsilon}
{}^{\lambda}Ae^{\lambda}\otimes_{\Bbbk}{}^{\mu}eA^{\mu}\otimes_A{}^{\upsilon}Ae^{\upsilon}
\ar[d]|-{\frac{1}{|G_e|^2}\big((\sigma\lambda^{\mone}\circ\mathbf{m})
\otimes\tau\upsilon^{\mone}\big)_{\tau\upsilon^{\mone}\in G_e,\sigma,\lambda,\mu\in G}}
\\
\displaystyle \bigoplus_{\xi,\theta}{}^{\xi}A^{\xi}\otimes_A{}^{\theta}Ae^{\theta}
\ar[rrrr]^{\frac{1}{|G_e||G|}(\sigma\xi^{\mone}\otimes \tau\theta^{\mone})_{\tau\theta^{\mone}\in G_e,
\sigma,\xi\in G}}
\ar[d]|-{\frac{1}{|G_e||G|}\big(\mathbf{m}\circ(\rho\xi^{\mone}\otimes\rho\theta^{\mone})
\big)_{\rho\theta^{\mone}\in G_e,\xi\in G}}
&&&&
\displaystyle \bigoplus_{\sigma,\tau}
{}^{\sigma}A^{\sigma}\otimes_A{}^{\tau}Ae^{\tau}
\ar[d]|-{\frac{1}{|G_e||G|}\big(\mathbf{m}\circ(\omega\sigma^{\mone}\otimes\omega\tau^{\mone})
\big)_{\omega\tau^{\mone}\in G_e,\sigma\in G}}
\\
\displaystyle \bigoplus_{\rho}{}^{\rho}Ae^{\rho}
\ar[rrrr]^{\frac{1}{|G_e|}(\omega\rho^{\mone})_{\omega\rho^{\mone}\in G_e}}
&&&&
\displaystyle \bigoplus_{\omega}
{}^{\omega}Ae^{\omega}
}
\end{displaymath}
}
where in the third horizontal arrow the conditions are
$\lambda\beta^{\mone},\mu\gamma^{\mone}, \upsilon\delta^{\mone}\in G_e$.

We want the $\rho$-component of the composition on the left hand side of the diagram to be
given by $\frac{1}{|G_e|}\rho$, if $\rho\in G_e$, and by zero otherwise. To see this, first
notice that multiplication $\mathbf{m}$ in the third map will give something non-zero
only if $\beta\gamma^{\mone}\in G_e$. Taking into account all conditions specified in the
diagram, this forces $\kappa,\beta,\gamma,\delta,\theta,\rho\in G_e$ in order for the $\rho$-component
to be non-zero. Each choice of  $\kappa,\beta,\gamma,\delta,\theta\in G_e$ and $\varphi,\xi\in G$
yields a summand $\frac{1}{|G|^2|G_e|^6}\rho$ in the composition (recall the factor
$\frac{1}{|G_e|}$ in \eqref{eq9}). Summing over all these possibilities hence produces the desired result.

The fact that the composition
\begin{displaymath}
\mathrm{G}\to\mathbbm{1}_{\mathtt{j}}\mathrm{G}\to \mathrm{G}\mathrm{F}\mathrm{G}\to
\mathrm{G}\mathbbm{1}_{\mathtt{i}}\to \mathrm{G}
\end{displaymath}
is the identity follows as above by flipping all tensor factors and replacing $Ae$ by $eA$ in appropriate places. This proves part~\eqref{prop61.1}.

Assume that $\mathtt{E}=\{e_1,e_2,\ldots,e_k\}$. For any $1\leq i\leq k$, we choose a Jordan-H\"{o}lder series of each
$Ae_i$ by
\begin{displaymath}
Ae_i=X_{i,0}\supsetneq X_{i,1}\supsetneq X_{i,2}\supsetneq \cdots \supsetneq X_{i,m_i}\supsetneq X_{i,m_i+1}=0.
\end{displaymath}
As our algebra $A$ is basic, each $\Bbbk$-space $X_{i,j}/X_{i,j+1}$, where $0\leq j\leq m_i$, is of dimension one.
For each $i$, we fix some basis $\mathtt{E}_i:=\{e_i, x_{i,j}: 1\leq j\leq m_i\}$ of $Ae_i$
such that we have $x_{i,j}\in X_{i,j}\setminus X_{i,j+1}$, for every $j$.
Then $\mathtt{A}:=\displaystyle\bigcup_{i=1}^{k}\mathtt{E}_i\supset\mathtt{E}$ is a basis of $A$.
Let $\mathbf{t}:A\to \Bbbk$ be the unique linear map such that, for all $a\in \mathtt{A}$, we have
\begin{displaymath}
\mathbf{t}(a) =
\begin{cases}
1, & a= x_{i,m_i},\text{ for some }i;\\
0, & \text{otherwise}.
\end{cases}
\end{displaymath}
For $a\in\mathtt{A}$, we denote by $a^*$ the unique element in $A$ which satisfies
\begin{displaymath}
\mathbf{t}(ba^*) =
\begin{cases}
1, & b=a;\\
0, & b\in \mathtt{A}\setminus\{a\}.
\end{cases}
\end{displaymath}

We now define the unit
\begin{displaymath}
\tilde{\eta}:(A,\tilde{\pi}_{1_{\hat{G}}})\to
(A\nu(e)\otimes_{\mathcal{Y}}eA,\varepsilon_{\nu(e)}\otimes_{\mathcal{Y}}\varepsilon_{e})
\end{displaymath}
by the vertical part of the diagram
\begin{displaymath}
\xymatrix{
A\ar[rrrr]^{\frac{1}{|G|}(\varphi)_{\varphi\in G}}
\ar[d]_{(\tilde{\eta}_{1}^{\alpha,\beta})_{\alpha\beta^{\mone}\in G_e}}
&&&&\displaystyle \bigoplus_{\varphi}{}^{\varphi}A^{\varphi}
\ar[d]^{(\tilde{\eta}_{\varphi}^{\gamma,\delta})_{\gamma\delta^{\mone}\in G_e,\varphi\in G}}\\
\displaystyle \bigoplus_{\alpha,\beta}{}^{\alpha}A\nu(e)^{\alpha}\otimes_{\Bbbk} {}^{\beta}eA^{\beta}
\ar[rrrr]^{\frac{1}{|G_e|^2}(\gamma\alpha^{\mone}\otimes\delta\beta^{\mone})_{\gamma\alpha^{\mone},
\delta\beta^{\mone}\in G_e}}
&&&&\displaystyle \bigoplus_{\gamma,\delta}{}^{\gamma}A\nu(e)^{\gamma}\otimes_{\Bbbk} {}^{\delta}eA^{\delta}
}
\end{displaymath}
where
\begin{equation}\label{eq10}
\tilde{\eta}_{\varphi}^{\gamma,\delta}(\varphi(1))=
\begin{cases}
\displaystyle
\frac{1}{|G_e|}\sum_{a\in\mathtt{A}}\gamma(a^*\nu(e))\otimes\delta(ea),&
\text{ if } \gamma\delta^{\mone}\in G_e;\\
0, & \text{ otherwise}.
\end{cases}
\end{equation}
Going right and then down, summing over $\varphi\in G$ cancels the scalar $\frac{1}{|G|}$
and hence the image of $1$ in the $(\gamma,\delta)$-component is given by the
right hand side of \eqref{eq10} and equals the $(\gamma,\delta)$-component of $(\tilde{\eta}_{1}^{\gamma,\delta}(1))_{\gamma\delta^{\mone}\in G_e}$,
which implies the first equality of ~\eqref{eq00}. Going down and then right, to obtain a non-zero contribution,
we need $\alpha\beta^{\mone},\gamma\alpha^{\mone},\delta\beta^{\mone}\in G_e$, which yields
$\delta\gamma^{\mone}\in G_e$. Summing over all such choices of $\alpha$ and $\beta$,
the image of $1$ in the $(\gamma,\delta)$-component is again given by the
right hand side of \eqref{eq10}, implying the second equality of \eqref{eq00}. Hence $\tilde{\eta}$ is well-defined.

Now we define the counit
\begin{displaymath}
\tilde{\epsilon}:(eA\otimes_{\mathcal{Y}}A\nu(e),\varepsilon_{e}\otimes_{\mathcal{Y}}\varepsilon_{\nu(e)})\to
(\Bbbk,\tilde{\pi}_{1_{\hat{G}}})
\end{displaymath}
by the vertical part of the diagram
\begin{displaymath}
\xymatrix{
\displaystyle \bigoplus_{\varphi}eA\otimes_{A} {}^{\varphi}A\nu(e)^{\varphi}
\ar[rrrr]^{\frac{1}{|G_{e}|^2}(\rho\otimes\psi\varphi^{\mone})_{\rho,\psi\varphi^{\mone}\in G_e}}
\ar[d]^{\frac{1}{|G_{e}|}(\alpha\circ\mathbf{t}\circ\mathbf{m})_{\alpha,\varphi\in G}}
&&&&
\displaystyle \bigoplus_{\rho,\psi}{}^{\rho} eA^{\rho}\otimes_{\Bbbk} {}^{\psi}A\nu(e)^{\psi}
\ar[d]_{\frac{1}{|G_{e}|}(\beta\circ\mathbf{t}\circ\mathbf{m}\circ
(\rho\otimes\rho)^{\mone})_{\beta,\psi,\rho\in G}}
\\
\displaystyle \bigoplus_{\alpha}{}^{\alpha}\Bbbk^{\alpha}
\ar[rrrr]^{\frac{1}{|G|}(\beta\alpha^{\mone})_{\alpha,\beta\in G}}
&&&& \displaystyle \bigoplus_{\beta}{}^{\beta}\Bbbk^{\beta}.
}
\end{displaymath}
Consider the $(\beta,\varphi)$-component of both compositions. First going down and then to the right,
the first map is zero, for $\varphi\not\in G_e$, as $\mathbf{t}(eAf)=0$ unless $f=\nu(e)$.
If $\varphi\in G_e$, then each $\alpha$ contributes $\frac{1}{|G||G_e|}\beta\circ\mathbf{t}\circ\mathbf{m}$.
Hence the resulting map is $\frac{1}{|G_e|}\beta\circ\mathbf{t}\circ\mathbf{m}$ and
the second equality of~\eqref{eq00} is satisfied.
The right vertical map
is zero unless $\psi\rho^{\mone}\in G_e$ which, together with the conditions on the upper horizontal map,
forces $\varphi,\rho,\psi\in G_e$. Summing over the choices for $\rho$ and $\psi$, we obtain the same
resulting map and thus the diagram commutes, in which case the first equality of \eqref{eq00} holds. Hence $\tilde{\epsilon}$ is well-defined.

We now verify the adjunction axioms. Denoting $(eA,\varepsilon_{e})$
by $\tilde{\mathrm{F}}$ and $(A\nu(e),\varepsilon_{\nu(e)})$ by $\tilde{\mathrm{G}}$, we need to show that
\begin{displaymath}
\tilde{\mathrm{F}}\to\tilde{\mathrm{F}}\mathbbm{1}_{\mathtt{i}}\to
\tilde{\mathrm{F}}\tilde{\mathrm{G}}\tilde{\mathrm{F}}\to
\mathbbm{1}_{\mathtt{j}}\tilde{\mathrm{F}}\to \tilde{\mathrm{F}}
\end{displaymath}
is the identity. To this end, we assemble our maps in a large commutative diagram
as in part~\eqref{prop61.1} and compute the left hand side. This is given by the composition
\begin{displaymath}
\xymatrix{
eA\ar[rrrr]^{\frac{1}{|G||G_e|}(\delta\otimes\varphi(1))_{\delta\in G_e,\varphi\in G}}
&&&&\displaystyle \bigoplus_{\delta,\varphi}{}^{\delta}eA^{\delta}\otimes_A{}^{\varphi}A^{\varphi}
\ar[d]_{\frac{1}{|G_e|}(\alpha\delta^{\mone}\otimes\tilde{\eta}_{\varphi}^{\beta,\gamma}
)_{\alpha\delta^{\mone},\beta\gamma^{\mone}\in G_e,\varphi\in G}}
\\
\displaystyle \bigoplus_{\xi,\theta}{}^{\xi}\Bbbk^{\xi}\otimes_{\Bbbk}{}^{\theta}eA^{\theta}
\ar[d]^{\frac{1}{|G_e||G|}\big(\mathbf{m}\circ(\omega\xi^{\mone}\otimes \omega\theta^{\mone})
\big)_{\omega\theta^{\mone}\in G_e,\xi\in G}}
&&&&
\displaystyle \bigoplus_{\alpha,\beta,\gamma}{}^{\alpha}eA^{\alpha}\otimes_A{}^{\beta}A\nu(e)^{\beta}
\otimes_{\Bbbk}{}^{\gamma}eA^{\gamma}
\ar[llll]_>>>>>>>>>>>>>>>>>>>>{\frac{1}{|G_e|^2}
\big((\xi\circ\mathbf{t}\circ\mathbf{m}\circ(\alpha\otimes\alpha)^{\mone})\otimes \theta\gamma^{\mone}\big)}
\\
\displaystyle \bigoplus_{\omega}{}^{\omega}eA^{\omega}
}
\end{displaymath}
where in the third map the conditions are $\theta\gamma^{\mone}\in G_e$ and $\alpha,\beta,\xi\in G$.
However, the third map is zero unless $\alpha\beta^{\mone}\in G_e$ (for the same reason involving $\mathbf{t}$
as used above), which, together with the other conditions in the diagram, shows that we have a non-zero
contribution to the $\omega$-component only if $\delta,\alpha,\beta,\gamma,\theta,\omega\in G_e$.
For $\omega\in G_e$, the contribution of a fixed choice of $\delta,\varphi,\alpha,\beta,\gamma,\xi,\theta$
to the image of $e$ is
\begin{displaymath}
\frac{1}{|G|^2|G_e|^6}\omega\left(\sum_{a\in\mathtt{A}}\mathbf{t}(e\beta\alpha^{\mone}(a^*\nu(e)))ea\right).
\end{displaymath}
Observing that, by our choice of $\mathtt{A}$, we have
\begin{displaymath}
\mathbf{t}(e\beta\alpha^{\mone}(a^*\nu(e)))=
\begin{cases}
1, & \text{ if } a=e;\\
0, & \text{ otherwise},
\end{cases}
\end{displaymath}
and summing over all choices of $\delta,\varphi,\alpha,\beta,\gamma,\xi,\theta$, we obtain that
the image of $e$ is $\frac{1}{|G_e|}\omega(e)$, as desired.

Now we verify the other axiom. The composition
\begin{displaymath}
\tilde{\mathrm{G}}\to\mathbbm{1}_{\mathtt{i}}\tilde{\mathrm{G}}\to
\tilde{\mathrm{G}}\tilde{\mathrm{F}}\tilde{\mathrm{G}}\to
\tilde{\mathrm{G}}\mathbbm{1}_{\mathtt{j}}\to \tilde{\mathrm{G}}
\end{displaymath}
is given by the diagram, which consists of the left hand side of a large commutative diagram
as in part~\eqref{prop61.1},
\begin{displaymath}
\xymatrix{
A\nu(e)\ar[rrrr]^{\frac{1}{|G||G_e|}(\alpha(1)\otimes\varphi)_{\alpha\in G,\varphi\in G_e}}
&&&&\displaystyle \bigoplus_{\alpha,\varphi}{}^{\alpha}A^{\alpha}\otimes_A{}^{\varphi}A\nu(e)^{\varphi}
\ar[d]_{\frac{1}{|G_e|}(\tilde{\eta}_{\alpha}^{\delta,\beta}\otimes\gamma\varphi^{\mone}
)_{\delta\beta^{\mone},\gamma\varphi^{\mone}\in G_e,\alpha\in G}}
\\
\displaystyle \bigoplus_{\xi,\theta}{}^{\theta}A\nu(e)^{\theta}\otimes_{\Bbbk}{}^{\xi}\Bbbk^{\xi}
\ar[d]^{\frac{1}{|G_e||G|}\big(\mathbf{m}\circ(\omega\theta^{\mone}\otimes \omega\xi^{\mone})
\big)_{\omega\theta^{\mone}\in G_e,\xi\in G}}
&&&&
\displaystyle \bigoplus_{\delta,\beta,\gamma}{}^{\delta}A\nu(e)^{\delta}\otimes_{\Bbbk}{}^{\beta}eA^{\beta}
\otimes_A{}^{\gamma}A\nu(e)^{\gamma}
\ar[llll]_>>>>>>>>>>>>>>>>>>>>{\frac{1}{|G_e|^2}
\big(\theta\delta^{\mone}\otimes(\xi\circ\mathbf{t}\circ\mathbf{m}\circ(\beta\otimes\beta)^{\mone})\big)}
\\
\displaystyle \bigoplus_{\omega}{}^{\omega}A\nu(e)^{\omega}
}
\end{displaymath}
where in the third map the conditions are $\theta\delta^{\mone}\in G_e$ and $\beta,\gamma,\xi\in G$.
Note that the third map is zero unless $\beta\gamma^{\mone}\in G_e$. Taking into account all conditions in the diagram,
this shows that we have a non-zero
contribution to the $\omega$-component only if $\varphi,\gamma,\beta,\delta,\theta,\omega\in G_e$.
For $\omega\in G_e$, the contribution of a fixed choice of $\alpha,\varphi,\gamma,\beta,\delta,\theta,\xi$
to the image of $\nu(e)$ is
\begin{displaymath}
\frac{1}{|G|^2|G_e|^6}\omega\left(\sum_{a\in\mathtt{A}}a^*\nu(e)\mathbf{t}(ea\beta^{\mone}\gamma(\nu(e)))\right).
\end{displaymath}
By our choice of $\mathtt{A}$, we have
\begin{displaymath}
\mathbf{t}(ea\beta^{\mone}\gamma(\nu(e)))=
\begin{cases}
1, & \text{ if } a^*=\nu(e);\\
0, & \text{ otherwise}.
\end{cases}
\end{displaymath}
Summing over all choices of $\alpha,\varphi,\gamma,\beta,\delta,\theta,\xi$, we obtain that
the image of $\nu(e)$ is $\frac{1}{|G_e|}\omega(\nu(e))$, as desired. This completes the proof.
\end{proof}

We now consider tensor products of indecomposable projective symmetric $B$-$B$-bi\-mo\-du\-les
with simple quotients of projective $A$-$\Bbbk$-bimodules. To this end, we extend our
notation to $G_{fe}:= G_{Af\otimes_{\Bbbk} eA}(=G_e\cap G_f)$, for $e,f\in \mathtt{E}$,
and denote the simple quotient of $(Ae, \varepsilon_e)$ by $(L_e, \varepsilon_e)$. As each $\varphi\in G$
is an automorphism of $A$, we have the induced action of $\varphi$ on $\{(L_e, \varepsilon_e)\,:\,e\in\mathtt{E}\}$
which maps each vector space $L_e$ to the vector space $L_{\varphi(e)}$.

\begin{lemma}\label{lem62}
In $\tilde{\mathcal{Y}}$, there is an isomorphism
$$(Af\otimes_\Bbbk eA, \varepsilon_{fe})\otimes_{\tilde{\mathcal{Y}}} (L_e, \varepsilon_e) \cong(Af, \frac{1}{|G_{fe}|} (\gamma )_{\gamma\in G_{fe}})\cong\bigoplus_{\xi}(Af, \varepsilon_{\xi}) $$
where $\xi$ runs over all characters appearing in the induction of the trivial $G_{fe}$-module to $G_f$.
\end{lemma}

\begin{proof}
We first construct an isomorphism between $(Af\otimes_{\Bbbk} eA, \varepsilon_{fe})\otimes_{\tilde{\mathcal{Y}}} (L_e, \varepsilon_e)$ and $(Af, \frac{1}{|G_{fe}|} (\gamma )_{\gamma\in G_{fe}})$.
In one direction, the morphism $\boldsymbol{g}$ is given by the diagram
{\small
\begin{displaymath}
\xymatrix{
Af \ar^{\frac{1}{|G_{fe}|}( \gamma )_{\gamma\in G_{fe}}}[rrrr]\ar|-{\frac{1}{|G_{fe}||G_{e}|}(\psi\otimes \psi(e)\otimes \varphi(l))_{\psi\in G_{fe},\varphi\in G_e}}[d]
&&&& \displaystyle\bigoplus_{\gamma} {}^{\gamma} Af^{\gamma}\ar|-{\hspace{-4mm}\frac{1}{|G_{fe}||G_{e}|}(\alpha\gamma^{\mone}\otimes \alpha(e)\otimes \beta(l))_{\alpha\gamma^{\mone}\in G_{fe},\beta\in G_e}}[d] \\
 \displaystyle\bigoplus_{\varphi,\psi}{}^{\psi}Af\otimes_{\Bbbk} eA^{\psi}\otimes_A {}^{\varphi}(L_e)^{\varphi}
\ar^{\frac{1}{|G_{fe}||G_{e}|}(\alpha\psi^{\mone}\otimes\alpha\psi^{\mone}\otimes \beta\varphi^{\mone} )}[rrrr]
&&&& \displaystyle\bigoplus_{\alpha,\beta}{}^{\alpha}Af\otimes_{\Bbbk} eA^{\alpha}\otimes_A {}^{\beta}(L_e)^{\beta}
}
\end{displaymath}
}

where the lower horizontal map is indexed by $\alpha\psi^{\mone}\in G_{fe}, \beta\varphi^{\mone}\in G_e $, and $l$ denotes the canonical generator of the one-dimensional module $L_e$ (the image of $e$ in $L_e$).
To see that the diagram commutes, first notice that the $(\alpha, \beta)$-component of the object in the lower right-hand corner is nonzero if and only if $\alpha\beta^{\mone}\in G_e$. If this is the case, then the $(\alpha, \beta)$-component of the map going first to the right and then down is given by
$$ f\mapsto
\begin{cases}
\frac{1}{|G_{fe}||G_{e}|}\alpha(f)\otimes \alpha(e)\otimes \beta(l), & \text{ if } \alpha \in G_{fe},\\
0, & \text{ otherwise.}
\end{cases}$$
First going down and then to the right, we notice that $\varphi\in G_e, \psi \in G_{fe}$ forces $\beta\in G_e, \alpha \in G_{fe}$, so we obtain the same result, which verifies \eqref{eq00}.

A morphism $\boldsymbol{h}$ in the other direction is given by
{\small
\begin{displaymath}
\xymatrix{
 \displaystyle\bigoplus_{\varphi}Af\otimes_{\Bbbk} eA\otimes_A {}^{\varphi}(L_e)^{\varphi}
\ar^{\frac{1}{|G_{fe}||G_{e}|}(\alpha\otimes\alpha\otimes \beta\varphi^{\mone} )_{\alpha\in G_{fe}, \beta\varphi^{\mone}\in G_e}}[rrrrr]\ar^{\frac{1}{|G_{fe}|}(\delta \otimes \mathbf{m})_{\delta\in G_{fe},\varphi\in G}}[d]
&&&&& \displaystyle\bigoplus_{\alpha,\beta}{}^{\alpha}Af\otimes_{\Bbbk} eA^{\alpha}\otimes_A {}^{\beta}(L_e)^{\beta} \ar_{\frac{1}{|G_{fe}|}(\gamma \alpha^{\mone} \otimes \mathbf{m}\circ (\alpha\otimes\alpha)^{\mone})_{\gamma \alpha^{\mone}\in G_{fe},\beta\in G}}[d] \\ \displaystyle\bigoplus
{}^{\delta}Af^{\delta} \ar^{\frac{1}{|G_{fe}|}( \gamma\delta^{\mone} )_{\gamma\delta^{\mone}\in G_{fe}}}[rrrrr]
&&&&& \displaystyle\bigoplus_{\gamma} {}^{\gamma} Af^{\gamma}
}
\end{displaymath}
}

A nonzero contribution to the $(\gamma, \varphi)$-component, when going first down and then to the right
can only happen for $\varphi\in G_e$ and $\gamma \in G_{fe}$, in which case the generator
$f\otimes e \otimes\varphi(l)$
(as an $A$-$\Bbbk$-bimodule) gets sent to $\frac{1}{|G_{fe}|}\gamma(f)$.
Similarly, going first to the right and then down, a nonzero contribution
only occurs for $\alpha\beta^{\mone}\in G_e$, which, together with the
conditions in the diagram, again forces $\beta, \varphi \in G_e, \alpha, \gamma \in G_{fe}$.
Hence, summing over such $\alpha$ and $\beta$, we obtain the same result which verifies \eqref{eq00}.

To check that both compositions of $\boldsymbol{g}$ and $\boldsymbol{h}$ are the respective identities
(that is, the correct idempotents), it suffices to consider the
compositions of the left hand side of the diagrams.

Starting with $\boldsymbol{g}\circ\boldsymbol{h}$ and considering
$$\bigoplus_{\varphi}Af\otimes_{\Bbbk} eA\otimes_A {}^{\varphi}(L_e)^{\varphi} \to \bigoplus_{\delta}
{}^{\delta}Af^{\delta} \to \bigoplus_{\alpha,\beta}{}^{\alpha}Af\otimes_{\Bbbk} eA^{\alpha}\otimes_A 
{}^{\beta}(L_e)^{\beta},$$
the $(\alpha,\beta,\varphi)$-component of the composition is zero unless
$\varphi, \beta \in G_e, \alpha\in G_{fe}$, in which case the generator
$f\otimes e \otimes\varphi(l)$ is mapped to
$\frac{1}{|G_{fe}||G_{e}|}\alpha(f)\otimes \alpha(e)\otimes \beta(l)$, as desired.

For $\boldsymbol{h}\circ\boldsymbol{g}$, we consider
$$Af \to \bigoplus_{\varphi,\psi}{}^{\psi}Af\otimes_{\Bbbk} eA^{\psi}\otimes_A {}^{\varphi}(L_e)^{\varphi} \to  \bigoplus_{\gamma} {}^{\gamma} Af^{\gamma}$$
and verify that, in the $\gamma$-component, $f$ is indeed sent to $\frac{1}{|G_{fe}|}\gamma(f)$, if $\gamma \in G_{fe}$, and zero otherwise, as claimed.

Hence we have an isomorphism $(Af\otimes_{\Bbbk} eA, \varepsilon_{fe})\otimes_{\tilde{\mathcal{Y}}} (L_e, \varepsilon_e) \cong(Af, \frac{1}{|G_{fe}|} (\gamma )_{\gamma\in G_{fe}})$, as stated.

Now notice that $\displaystyle\frac{1}{|G_{fe}|} \sum_{\gamma \in G_{fe}}\gamma$ is a trivial idempotent on $G_{fe}$. When viewed as an idempotent of the larger group $G_f$, it decomposes into precisely the (multiplicity-free) sum of those idempotents affording characters $\xi$ of $G_f$ which appear in the induction of the trivial character from $G_{fe}$ to $G_f$. This proves the proposition.
\end{proof}

\begin{proposition}\label{prop63}
We have adjunctions $((Af\otimes_{\Bbbk}eA, \varepsilon_{fe}), (A\nu(e)\otimes_{\Bbbk}fA, \varepsilon_{\nu(e)f}))$,
for idempotents $e,f\in A$.
\end{proposition}

\begin{proof}
From the defining action of $\cG_B$ on $B$-$\mathrm{mod}$ we have that
$(Af\otimes_{\Bbbk}eA, \varepsilon_{fe})$ is left adjoint to
$(A\nu(e)\otimes_{\Bbbk}fA, \varepsilon_{\chi})$, for some $\chi\in \hat{G}_{\nu(e)f}$.
We thus have an isomorphism of nonzero spaces of homomorphisms,
{\small
\begin{equation*}
\begin{split}
\mathrm{Hom}_{\tilde{\mathcal{Y}}}((Af, \frac{1}{|G_{fe}|} (\beta)_{\beta\in G_{fe}}),(L_f, \varepsilon_f))
&\cong  \mathrm{Hom}_{\tilde{\mathcal{Y}}}((Af\otimes_{\Bbbk}eA,\varepsilon_{fe})\otimes_{\tilde{\mathcal{Y}}}(L_e, \varepsilon_e), (L_f, \varepsilon_f))\\
&\cong \mathrm{Hom}_{\tilde{\mathcal{Y}}}((L_e, \varepsilon_e),(A\nu(e)\otimes_{\Bbbk}fA, \varepsilon_{\chi})\otimes_{\tilde{\mathcal{Y}}} (L_f, \varepsilon_f)),
\end{split}
\end{equation*}
}

where the first isomorphism follows from Lemma~\ref{lem62}.

By (the opposite of) Proposition \ref{propmultid2}, noting that $G_{fe} = G_{\nu(e)f}$ and $G_{\nu(e)}=G_e$, there are isomorphisms
\begin{equation*}
\begin{split}
(A\nu(e)\otimes_{\Bbbk}fA, \varepsilon_{\chi})\otimes_{\tilde{\mathcal{Y}}} (L_f, \varepsilon_f)) &\cong (A, \tilde{\pi}_\chi)\otimes_{\tilde{\mathcal{Y}}} (A\nu(e)\otimes_{\Bbbk}fA, \varepsilon_{\nu(e)f})\otimes_{\tilde{\mathcal{Y}}} (L_f, \varepsilon_f)\\
&\cong (A, \tilde{\pi}_\chi)\otimes_{\tilde{\mathcal{Y}}} (A\nu(e), \frac{1}{|G_{fe}|} (\gamma)_{\gamma\in G_{fe}})\\
&\cong  (A\nu(e),  \frac{1}{|G_{fe}|} (\chi(\gamma)\gamma)_{\gamma\in G_{fe}}),
\end{split}
\end{equation*}
yielding the isomorphism
\begin{displaymath}
\mathrm{Hom}_{\tilde{\mathcal{Y}}}((Af, \frac{1}{|G_{fe}|} (\beta)_{\beta\in G_{fe}}),(L_f, \varepsilon_f))\cong
\mathrm{Hom}_{\tilde{\mathcal{Y}}}((L_e, \varepsilon_e),(A\nu(e),  \frac{1}{|G_{fe}|} (\chi(\gamma)\gamma)_{\gamma\in G_{fe}})).
\end{displaymath}
Denote the left hand side by $U$ and the right hand side by $V$.

Now we claim that $\dim_\Bbbk U=1$. By \eqref{eq00}, for any morphism $\boldsymbol{g}\in U$, we have
\begin{equation}\label{eq11}
\text{id}_{(L_f, \varepsilon_f)}\circ \boldsymbol{g}= \boldsymbol{g}=\boldsymbol{g}\circ \text{id}_{(Af, \frac{1}{|G_{fe}|} (\beta)_{\beta\in G_{fe}})}.
\end{equation}
Assume that $\boldsymbol{g}=(k_\alpha\alpha)_{\alpha\in G_f}$, where $k_\alpha\in\Bbbk$, and consider the diagram
\begin{displaymath}
\xymatrix{Af\ar[rrr]^{ \frac{1}{|G_{fe}|} (\beta)_{\beta\in G_{fe}}}
\ar[d]_{(k_{\alpha}\alpha)_{\alpha\in G_f} }&&& \displaystyle\bigoplus_{\beta}
{}^{\beta}Af^{\beta} \ar[d]^{(k_{\gamma\beta^{\mone}}\gamma\beta^{\mone})_{\gamma\beta^{\mone}\in G_{f}}}\\
\displaystyle\bigoplus_{\alpha}{}^{\alpha}(L_f)^{\alpha}\ar[rrr]^{\frac{1}{|G_{f}|} (\gamma\alpha^{\mone})_{\gamma\alpha^{\mone}\in G_{f}}} &&&  \displaystyle\bigoplus_{\gamma}{}^{\gamma}(L_f)^{\gamma}.}
\end{displaymath}
The morphism $\text{id}_{(L_f,\varepsilon_f)}\circ \boldsymbol{g}$ is exactly the path going down and then right,
taking into account the fact that $\alpha\in G_f$ forces $\gamma\in G_f$, and the $\gamma$-component of this morphism is given by
$$ f\mapsto
\begin{cases}
\frac{1}{|G_{f}|}\Big(\displaystyle\sum_{\alpha\in G_f}k_{\alpha}\Big) \gamma(f),& \text{ if } \gamma \in G_{f};\\
0, & \text{ otherwise.}
\end{cases}$$
The first equality of \eqref{eq11} shows that
$\frac{1}{|G_{f}|}\displaystyle\sum_{\alpha\in G_f}k_{\alpha}=k_{\gamma}$,
for all $\gamma\in G_f$, and hence
$k_{\gamma}=k_{\gamma'}$ for all $\gamma,\gamma'\in G_f$.
Then we obtain $\boldsymbol{g}=(k\alpha)_{\alpha\in G_f}$, where $k\in \Bbbk$,
and the first equality is automatically satisfied.
Going right and then down and using the fact that
$\beta\in G_{fe},\gamma\beta^{\mone}\in G_f$ implies $\gamma\in G_f$,
the second equality of \eqref{eq11} is easily verified.
The claim follows.

As $U\cong V$, we have $\dim_\Bbbk V=1$. Using \eqref{eq00}, for any morphism $\boldsymbol{h}\in V$, we have
\begin{equation}\label{eq12}
\text{id}_{(A\nu(e),  \frac{1}{|G_{fe}|} (\chi(\gamma)\gamma)_{\gamma\in G_{fe}})}\circ \boldsymbol{h}= \boldsymbol{h}=\boldsymbol{h}\circ \text{id}_{(L_e, \varepsilon_e)}.
\end{equation}
Assume that $\boldsymbol{h}=(l_\alpha\alpha)_{\alpha\in G_e}$, where $l_\alpha\in\Bbbk$, and consider the diagram
\begin{displaymath}
\xymatrix{L_e\ar[rrrr]^{ \frac{1}{|G_{e}|} (\beta)_{\beta\in G_{e}}}
\ar[d]_{(l_{\alpha}\alpha)_{\alpha\in G_e} }&&&& \displaystyle\bigoplus_{\beta}
{}^{\beta}(L_e)^{\beta} \ar[d]^{(l_{\delta\beta^{\mone}}\delta\beta^{\mone})_{\delta\beta^{\mone}\in G_{e}}}\\
\displaystyle\bigoplus_{\alpha}{}^{\alpha}A\nu(e)^{\alpha}\ar[rrrr]^{\frac{1}{|G_{fe}|} (\chi(\delta\alpha^{\mone})\delta\alpha^{\mone})_{\delta\alpha^{\mone}\in G_{fe}}} &&&&  \displaystyle\bigoplus_{\delta}{}^{\delta}A\nu(e)^{\delta}.}
\end{displaymath}
The morphism $\boldsymbol{h}\circ \text{id}_{(L_e, \varepsilon_e)}$ coincides with the path going to the
right and then down. Note that $\beta\in G_e,\delta\beta^{\mone}\in G_e$ forces $\delta\in G_e$.
Then the $\delta$-component of this composition is given by
$$ e\mapsto
\begin{cases}
\frac{1}{|G_{e}|}\Big(\displaystyle\sum_{\beta\in G_e}l_{\delta\beta^{\mone}}\Big)
\delta(e),& \text{ if } \delta \in G_{e};\\
0, & \text{ otherwise.}
\end{cases}$$
By re-indexing, the second equality of \eqref{eq12} shows that, for all $\delta\in G_e$,
we have $\frac{1}{|G_{e}|}\displaystyle\sum_{\sigma\in G_e}l_{\sigma}=l_{\delta}$, and thus
$l_{\delta}=l_{\delta'}$, for all $\delta,\delta'\in G_e$.
Therefore we have $\boldsymbol{h}=(l\alpha)_{\alpha\in G_e}$, where $l\in \Bbbk$, and the second equality holds.
Due to $\dim_\Bbbk V=1$, the first equality of \eqref{eq12} should also hold for any $l\in \Bbbk^{\ast}$.
Going down and then to the right, the $\delta$-component of
$\text{id}_{(A\nu(e),  \frac{1}{|G_{fe}|} (\chi(\gamma)\gamma)_{\gamma\in G_{fe}})}\circ \boldsymbol{h}$ is
zero unless $\delta\in G_e$,
in which case $e$ is sent to
\begin{displaymath}
\frac{1}{|G_{fe}|}\Big(\displaystyle\sum_{\alpha\in\delta G_{fe}}\chi(\delta\alpha^{\mone})\Big)l\delta(e)=\frac{1}{|G_{fe}|}\Big(\displaystyle\sum_{\varphi\in G_{fe}}\chi(\varphi)\Big)l\delta(e).
\end{displaymath}
The first equality implies that $\frac{1}{|G_{fe}|}\Big(\displaystyle\sum_{\varphi\in G_{fe}}\chi(\varphi)\Big)=1$.
By multiplying any $\chi(\gamma)$, where $\gamma\in G_{fe}$, to both side of the latter, we obtain
$$\chi(\gamma)=\frac{1}{|G_{fe}|}\Big(\displaystyle\sum_{\varphi\in G_{fe}}\chi(\gamma\varphi)\Big)=\frac{1}{|G_{fe}|}\Big(\displaystyle\sum_{\psi\in G_{fe}}\chi(\psi)\Big)=1.$$
Therefore $\chi=\varepsilon_{\nu(e)f}$ and the proof is complete.
\end{proof}

\begin{proposition}\label{prop64}
In $\cG_B$, we have adjunctions $((A_i, \tilde{\pi}_{\chi}),(A_i, \tilde{\pi}_{\chi^{\mone}}))$,
for each $\chi\in\hat{G}$ and $i=1,\dots, n$. Similarly, we have adjunction
$((\Bbbk, \tilde{\pi}_{\chi}),(\Bbbk, \tilde{\pi}_{\chi^{\mone}}))$, for each $\chi\in\hat{G}$.
\end{proposition}

\begin{proof}
By Proposition \ref{propmultid}, we have 
\begin{displaymath}
(A_i, \tilde{\pi}_{\chi})\otimes_{\tilde{\mathcal{Y}}}(A_i, \tilde{\pi}_{\chi^{\mone}}) 
\cong  (A_i, \tilde{\pi}_{1_{\hat{G}}})\cong 
(A_i, \tilde{\pi}_{\chi^{\mone}})\otimes_{\tilde{\mathcal{Y}}}(A_i, \tilde{\pi}_{\chi})
\end{displaymath}
and similarly for $\Bbbk$. Both unit and counit are then just identities and the claim is immediate.
\end{proof}

\begin{proposition}\label{prop65}
If $A$ is self-injective, then the $2$-categories $\cG_{A}$ and $\cG_B$ are weakly fiat.
If $A$ is weakly symmetric, then both $\cG_{A}$ and $\cG_B$ are fiat.
\end{proposition}

\begin{proof}
Assume $A$ is self-injective. Proposition \ref{propmultid2} shows that any indecomposable $1$-mor\-phism
can be written as a product of those treated in Propositions \ref{prop61}, \ref{prop63} and \ref{prop64}.
This implies that $\cG_{A}$ and $\cG_B$ are weakly fiat. If $A$ is weakly symmetric, then $\nu$ is the identity,
and all adjunctions given in Propositions \ref{prop61}, \ref{prop63} and \ref{prop64} become (weakly)
involutive, proving fiatness.
\end{proof}

\subsection{Simple transitive $2$-representations of $\cG_{A}$}\label{s3.5}
Now we can formulate our first main result. We assume that $A$ is weakly symmetric, basic and there is a fixed complete $G$-invariant
set $\mathtt{E}$ of primitive idempotents, so that $\cG_{A}$ and $\cG_{B}$ are fiat.

\begin{theorem}\label{thm11}
Under the above assumptions, for every two-sided cell $\mathcal{J}$ in $\cG_{A}$, there is a natural bijection between
equivalence classes of simple transitive $2$-representations of $\cG_{A}$ with apex $\mathcal{J}$
and pairs $(K,\omega)$, where $K$ is a subgroup of $G$ and $\omega\in H^2(K,\Bbbk^*)$.
\end{theorem}

\begin{proof}
For $\mathcal{J}=\mathcal{J}_{\mathtt{i}}$, where $i=1,2,\dots,n$, 
Proposition~\ref{propmultid} shows that the $\mathcal{J}$-simple quotient of $\cG_{A}$
is biequivalent to the $2$-category $\mathrm{Rep}(G)$ from \cite{Os}.
Therefore the statement follows from \cite[Theorem~2]{Os}.

For $\mathcal{J}=\mathcal{J}_0$, consider $B=A\times\Bbbk$. Then we can realize $\cG_{A}$ as a
both $1$- and $2$-full subcategory of $\cG_{B}$ in the obvious way. Let $\mathtt{j}$ denote the
object corresponding to the additional factor $\Bbbk$. Let $\mathcal{H}_1$ be the
$\mathcal{H}$-class in $\cG_{B}$ containing the identity $1$-morphism on $\Bbbk$. By
Proposition~\ref{propmultid}, $\mathcal{H}_1$ contains $|G|$ many $1$-morphisms, moreover,
$\mathcal{H}_1$ is contained in $\mathcal{J}_0^{(B)}$, the two-sided cell of projective bimodules in
$\cG_{B}$.
Note that the $1$- and $2$-full $2$-subcategory $\cA_{\mathcal{H}_1}$ of $\cG_{B}$ with object $\mathtt{j}$ is
biequivalent to the $2$-category $\mathrm{Rep}(G)$ as above. Hence, by \cite[Theorem~2]{Os},
there is a natural bijection between equivalence classes of simple transitive $2$-representations of
$\cA_{\mathcal{H}_1}$ with apex $\mathcal{H}_1$ and pairs $(K,\omega)$ as in the theorem. By \cite[Theorem~15]{MMMZ},
there is a bijection  between equivalence classes of simple transitive $2$-representations of
$\cA_{\mathcal{H}_1}$ and equivalence classes of simple transitive $2$-representations of $\cG_{B}$
with apex $\mathcal{J}_0^{(B)}$.

Let $\mathcal{H}_2$ be any self-dual $\mathcal{H}$-class in $\cG_{A}$ contained in $\mathcal{J}_0$.
Notice that this is also a  self-dual $\mathcal{H}$-class in $\cG_{B}$ contained in $\mathcal{J}_0^{(B)}$.
Let $\cA_{\mathcal{H}_2}$ be the corresponding $1$- and $2$-full $2$-subcategory of $\cG_{A}$
(and of $\cG_{B}$), cf. \cite[Subsection~4.2]{MMMZ}.  By \cite[Theorem~15]{MMMZ},
there is a bijection between equivalence classes of simple transitive $2$-representations of
$\cA_{\mathcal{H}_2}$ and equivalence classes of simple transitive $2$-representations of $\cG_{B}$
with apex $\mathcal{J}_0^{(B)}$, and also of $\cG_{A}$
with apex $\mathcal{J}_0$. The claim follows.
\end{proof}

To prove Theorem~\ref{thm11}, one could alternatively use \cite[Corollary~12]{MMMT}.

\begin{remark}
{\em
An analogue of Theorem~\ref{thm11} is also true in the weakly fiat case, that is when  $A$
is just self-injective but not necessarily weakly symmetric. However, the proof requires an adjustment
of the results of \cite[Theorem~15]{MMMZ} to the case when instead of one diagonal $\mathcal{H}$-cell
one considers a diagonal block which is stable under $\star$. One could carefully go through the proof
\cite[Theorem~15]{MMMZ} and check that everything works.
}
\end{remark}

\subsection{A class of examples}\label{s3.7}

Fix a positive integer $n>1$ and let $A$ be the quotient of the path algebra of the cyclic quiver
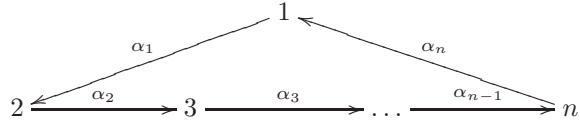
\begin{displaymath}
\xymatrix{
&&&1\ar[llld]_{\alpha_1}&&&\\
2\ar[rr]^{\alpha_2}&&3\ar[rr]^{\alpha_3}&&\dots\ar[rr]^{\alpha_{n-1}}&&n\ar[lllu]_{\alpha_{n}}
}
\end{displaymath}
modulo the ideal generated by all paths of length $n$. Now we let $G$ be the cyclic group of
order $n$ whose generator $\varphi$ acts on $A$ by sending $e_i$ to $e_{i+1}$ and 
$\alpha_i$ to $\alpha_{i+1}$ (where we compute indices modulo $n$). 

For $i=1,2,\dots,n$, we denote by $\mathrm{F}_i$ the indecomposable $1$-morphism in $\cG_A$ 
corresponding to tensoring with $Ae_1\otimes_{\Bbbk}e_iA$ (we omit the idempotents since the 
action of $G$ is free). Then $(\mathrm{F}_i,\mathrm{F}_{n+1-i})$ is an adjoint pair 
(and, indeed, biadjoint), for each $i$. Hence $\cG_A$ is fiat.

Note that every subgroup of $G$ is cyclic and $H^2(\mathbb{Z}/k\mathbb{Z},\Bbbk^*)\cong 
\Bbbk^*/(\Bbbk^*)^k\cong \{e\}$ since $\Bbbk$ is algebraically closed. Therefore, 
simple transitive $2$-representations of $\cG_A$ are in bijection with divisors of $n$.
For $d\vert n$, the algebra underlying the simple transitive $2$-representations of $\cG_A$
corresponding to $d$ is the algebra $A^{\langle\varphi^{\frac{n}{d}}\rangle}$ with the
obvious action of $\cG_A$.

\section{Two-element $\mathcal{H}$-cells with no self-adjoint elements}\label{s7}

\subsection{Basic combinatorics}\label{s7.1}

\begin{proposition}\label{prop71}
Let $\cC$ be a fiat $2$-category such that
\begin{itemize}
\item $\cC$ has one object $\mathtt{i}$;
\item $\cC$ has two two-sided cells, each of which is also a right cell and a left cell, one being $\{\mathbbm{1}_{\mathtt{i}}\}$
and the other one given by $\{\mathrm{F},\mathrm{G}\}$ with $\mathrm{F}\not\cong\mathrm{G}$;
\item $\mathrm{F}^{\star}\cong\mathrm{G}$.
\end{itemize}
Then there exists $n\in\mathbb{Z}_{>0}$ such that
\begin{equation}\label{eq7-1}
\mathrm{F}\mathrm{F}\cong \mathrm{F}\mathrm{G}\cong \mathrm{G}\mathrm{F}
\cong \mathrm{G}\mathrm{G}\cong (\mathrm{F}\oplus \mathrm{G})^{\oplus n}.
\end{equation}
\end{proposition}

\begin{proof}
We have
\begin{gather*}
\mathrm{F}\mathrm{F}\cong \mathrm{F}^{\oplus a_1}\oplus \mathrm{G}^{\oplus a_2},\qquad
\mathrm{F}\mathrm{G}\cong \mathrm{F}^{\oplus b_1}\oplus \mathrm{G}^{\oplus b_2},\\
\mathrm{G}\mathrm{F}\cong \mathrm{F}^{\oplus c_1}\oplus \mathrm{G}^{\oplus c_2},\qquad
\mathrm{G}\mathrm{G}\cong \mathrm{F}^{\oplus d_1}\oplus \mathrm{G}^{\oplus d_2},
\end{gather*}
for some $a_1,a_2,b_1,b_2,c_1,c_2,d_1,d_2\in\mathbb{Z}_{\geq 0}$.

From $\mathrm{F}^{\star}\cong\mathrm{G}$, we see that $(\mathrm{F}\mathrm{G})^{\star}\cong\mathrm{F}\mathrm{G}$
and $(\mathrm{G}\mathrm{F})^{\star}\cong\mathrm{G}\mathrm{F}$. This implies $b_1=b_2=:b$
and $c_1=c_2=:c$. Furthermore, $(\mathrm{F}\mathrm{F})^{\star}\cong\mathrm{G}\mathrm{G}$, which implies $a_1=d_2=:x$ and $a_2=d_1=:y$.

As $\mathrm{G}$ is in the same left cell as  $\mathrm{F}$, we obtain $y+c>0$ and $y+b>0$.

{\bf Case 1: $y=0$.} In this case we have $c,b>0$ by the above, and
\begin{displaymath}
\mathrm{F}\mathrm{F}\cong \mathrm{F}^{\oplus x},\quad
\mathrm{F}\mathrm{G}\cong \mathrm{F}^{\oplus b}\oplus \mathrm{G}^{\oplus b},\quad
\mathrm{G}\mathrm{F}\cong \mathrm{F}^{\oplus c}\oplus \mathrm{G}^{\oplus c},\quad
\mathrm{G}\mathrm{G}\cong \mathrm{G}^{\oplus x}.
\end{displaymath}
We use this to compute both sides of the isomorphism
$(\mathrm{F}\mathrm{G})\mathrm{G}\cong\mathrm{F}(\mathrm{G}\mathrm{G})$. This
yields  $b^2=xb$ (by comparing the coefficients as $\mathrm{F}$) and $b^2+xb=xb$
(by comparing the coefficients as $\mathrm{G}$). Hence $b=0$, a contradiction.
Therefore this case cannot occur.

{\bf Case 2: $y>0$.} In this case we have
\begin{displaymath}
\mathrm{F}\mathrm{F}\cong \mathrm{F}^{\oplus x}\oplus \mathrm{G}^{\oplus y},\quad
\mathrm{F}\mathrm{G}\cong \mathrm{F}^{\oplus b}\oplus \mathrm{G}^{\oplus b},\quad
\mathrm{G}\mathrm{F}\cong \mathrm{F}^{\oplus c}\oplus \mathrm{G}^{\oplus c},\quad
\mathrm{G}\mathrm{G}\cong \mathrm{F}^{\oplus y}\oplus \mathrm{G}^{\oplus x}.
\end{displaymath}
We use this to compute both sides of the isomorphism
$(\mathrm{F}\mathrm{G})\mathrm{F}\cong\mathrm{F}(\mathrm{G}\mathrm{F})$, and obtain  $xc=xb$ (by comparing the coefficients as $\mathrm{F}$) and $yc=yb$
(by comparing the coefficients as $\mathrm{G}$). As $y>0$, we have  $c=b$.

Finally, we compute both sides of the isomorphism
$(\mathrm{F}\mathrm{G})\mathrm{G}\cong\mathrm{F}(\mathrm{G}\mathrm{G})$. This
implies  $b^2+by=xy+bx$ (by comparing the coefficients as $\mathrm{F}$) and $b^2+bx=y^2+bx$
(by comparing the coefficients as $\mathrm{G}$). As $y>0$ and $b\geq 0$, from the second equation
we deduce $b=y$. Using $y>0$ and $b=y$, the first equation yields $b=x$. The claim follows.
\end{proof}

\subsection{The algebra of the cell $2$-representation}\label{s7.2}

Let $\cC$ be a fiat $2$-category as in Proposition~\ref{prop71}.
Consider the cell $2$-representation $\mathbf{C}_{\mathcal{H}}$ of $\cC$,
where $\mathcal{H}=\{\mathrm{F},\mathrm{G}\}$.
Denote by $A$ its underlying basic algebra with a fixed
decomposition $1_A =e_{\mathrm{F}} +e_{\mathrm{G}}$ of the identity into primitive orthogonal idempotents.
Let $P_{\mathrm{F}}$ and $P_{\mathrm{G}}$ denote the corresponding indecomposable projective $A$-modules
and  $L_{\mathrm{F}}$ and $L_{\mathrm{G}}$ their respective simple tops. Note that fiatness of $\cC$
implies self-injectivity of $A$ (cf. \cite[Theorem~2]{KMMZ}).

From \eqref{eq7-1} we obtain that the matrix describing the action of both
$\mathrm{F}$ and $\mathrm{G}$ in the cell $2$-representation (in the basis of indecomposable projective modules) is
\begin{equation}\label{eq7-2}
\left(\begin{array}{cc}
n&n\\
n&n
\end{array}\right).
\end{equation}
By \cite[Lemma 10]{MM5}, the same matrix describes the action of both
$\mathrm{F}$ and $\mathrm{G}$ in the abelianization of the cell $2$-representation
in the basis of simple modules. Without loss of generality,
assume that $\mathrm{G}$ is the Duflo involution of the left cell $\mathcal{H}$.
In the cell $2$-representation,
we then have $\mathrm{G}L_{\mathrm{G}} \cong P_{\mathrm{G}}$ and $\mathrm{F}L_{\mathrm{G}} \cong P_{\mathrm{F}}$. This, together with the description of the matrix of the action in \eqref{eq7-2}, shows that
\begin{equation}\label{eq13}
[P_{\mathrm{G}}: L_{\mathrm{F}}]=[P_{\mathrm{F}}: L_{\mathrm{F}}]=
[P_{\mathrm{G}}: L_{\mathrm{G}}]=[P_{\mathrm{F}}: L_{\mathrm{G}}]=n.
\end{equation}
Therefore, the Cartan matrix of $A$ is given by \eqref{eq7-2}.

As the bimodules $X$ and $Y$, representing $\mathrm{F}$ and $\mathrm{G}$, respectively, are projective,
see \cite[Theorem 2]{KMMZ} and \cite[Lemma 13]{MM5} for details,
we deduce that $Ae_{\mathrm{F}}\otimes_{\Bbbk} e_{\mathrm{G}}A$ appears as a direct summand of $X$
and $Ae_{\mathrm{G}}\otimes_{\Bbbk} e_{\mathrm{G}}A$ appears as a direct summand of $Y$.
Due to $\mathrm{G}^\star\cong \mathrm{F}$, and
\begin{displaymath}
0\neq \mathrm{Hom}_{\mathbf{C}_{\mathcal{H}}}(\mathrm{G} L_{\mathrm{G}}, L_{\mathrm{G}}) \cong \mathrm{Hom}_{\mathbf{C}_{\mathcal{H}}}( L_{\mathrm{G}},\mathrm{F} L_{\mathrm{G}})\cong \mathrm{Hom}_{\mathbf{C}_{\mathcal{H}}}( L_{\mathrm{G}},P_{\mathrm{F}}),
\end{displaymath}
the algebra $A$ is not weakly symmetric.
Furthermore, we have
\begin{displaymath}
0=\mathrm{Hom}_A(P_{\mathrm{G}},L_{\mathrm{F}} ) \cong\mathrm{Hom}_A(\mathrm{G}L_{\mathrm{G}},L_{\mathrm{F}} )
\cong \mathrm{Hom}_A(L_{\mathrm{G}},\mathrm{F}L_{\mathrm{F}} ),
\end{displaymath}
so  $\mathrm{F}L_{\mathrm{F}}$ is a direct sum of copies of $P_G$.
Comparing the Cartan matrix of $A$ with the matrix of the action of $\mathrm{F}$ in the basis of simples
(both given by  \eqref{eq7-2}), we see that $\mathrm{F}L_{\mathrm{F}}\cong P_{\mathrm{G}}$.
Similarly we deduce $\mathrm{G}L_{\mathrm{F}}\cong P_{\mathrm{F}}$.
Hence, we have
\begin{equation}\label{eq131}
X\cong Ae_{\mathrm{F}}\otimes_{\Bbbk} e_{\mathrm{G}}A\oplus Ae_{\mathrm{G}}\otimes_{\Bbbk} e_{\mathrm{F}}A \quad \text{and}\quad
Y\cong Ae_{\mathrm{F}}\otimes_{\Bbbk} e_{\mathrm{F}}A\oplus Ae_{\mathrm{G}}\otimes_{\Bbbk} e_{\mathrm{G}}A.
\end{equation}

\subsection{Functors isomorphic to the identity endomorphism of $2$-representations}\label{s3.7}

In this subsection, we will formulate a general result for arbitrary finitary
$2$-category $\cC$. This result will be needed for Subsection~\ref{s7.3}. For simplicity, we assume that
$\cC$ has only one object $\mathtt{i}$. Let $\mathbf{M}$ be a finitary $2$-representation of $\cC$.
Let $(\mathrm{Id}_{\mathbf{M}},\eta):\mathbf{M}\to \mathbf{M}$ be the identity endomorphism $\mathbf{M}$.
Here $\eta$ is given by the family $\{\eta_{\mathrm{F}},\, \mathrm{F}\in\cC(\mathtt{i},\mathtt{i})\}$
of natural transformations where each $\eta_{\mathrm{F}}$ is the identity natural transformation of
$\mathbf{M}(\mathrm{F})$.

\begin{lemma}\label{lem333}
Let $\Phi:\mathbf{M}(\mathtt{i})\to \mathbf{M}(\mathtt{i})$ 
be a functor isomorphic to the identity functor $\mathrm{Id}_{\mathbf{M}(\mathtt{i})}$.
Then there exists a family of natural isomorphisms
$\{\zeta_{\mathrm{F}},\, \mathrm{F}\in\cC(\mathtt{i},\mathtt{i})\}=:\zeta$ such that
$(\Phi,\zeta)$ is an endomorphism of the $2$-representation $\mathbf{M}$.
\end{lemma}

\begin{proof}
Note that $\Phi \cong \mathrm{Id}_{\mathbf{M}(\mathtt{i})}$ as a functor.
Let $\theta:\mathrm{Id}_{\mathbf{M}(\mathtt{i})} \to \Phi$ be a fixed natural isomorphism
and set $\nu:=\theta^{\mone}$. For any $1$-morphisms $\mathrm{F},\mathrm{G}$ and
$2$-morphism $\alpha:\mathrm{F}\to\mathrm{G}$, consider the diagram
\begin{equation}\label{eq371}
\xymatrix{\Phi\circ\mathbf{M}(\mathrm{F})\ar[rr]^{\nu\circ_{\mathrm{h}}\mathrm{id}_{\mathbf{M}(\mathrm{F})}}
\ar[d]_{\mathrm{id}_{\Phi}\circ_{\mathrm{h}}\mathbf{M}(\alpha)}&& \mathbf{M}(\mathrm{F})\ar[rr]^{\mathrm{id}_{\mathbf{M}(\mathrm{F})}\circ_{\mathrm{h}}\theta}
\ar[d]_{\mathbf{M}(\alpha)}&&\mathbf{M}(\mathrm{F})\circ\Phi\ar[d]^{\mathbf{M}(\alpha)\circ_{\mathrm{h}}\mathrm{id}_{\Phi}}\\
\Phi\circ\mathbf{M}(\mathrm{G})\ar[rr]^{\nu\circ_{\mathrm{h}}\mathrm{id}_{\mathbf{M}(\mathrm{G})}}&& \mathbf{M}(\mathrm{G})\ar[rr]^{\mathrm{id}_{\mathbf{M}(\mathrm{G})}\circ_{\mathrm{h}}\theta}&&\mathbf{M}(\mathrm{G})\circ\Phi.
}
\end{equation}
Here in the middle column we use the fact that, for any $1$-morphism $\mathrm{H}$, we have
\begin{displaymath}
\mathrm{Id}_{\mathbf{M}(\mathtt{i})}\circ_{\mathrm{h}}\mathbf{M}(\mathrm{H})=\mathbf{M}(\mathrm{H})
=\mathbf{M}(\mathrm{H})\circ_{\mathrm{h}}\mathrm{Id}_{\mathbf{M}(\mathtt{i})}.
\end{displaymath}
Diagram~\eqref{eq371} commutes thanks to the interchange law, indeed,
both paths in the left square are equal to $\nu\circ_{\mathrm{h}}\mathbf{M}(\alpha)$
and both paths  in the right square are equal to $\mathbf{M}(\alpha)\circ_{\mathrm{h}}\theta$.

For each $1$-morphism ${\mathrm{F}}$, define
\begin{equation}\label{eq335}
\begin{split}
\zeta_{\mathrm{F}}:&=(\mathrm{id}_{\mathbf{M}(\mathrm{F})}\circ_{\mathrm{h}}\theta)
\circ_{\mathrm{v}}(\nu\circ_{\mathrm{h}}\mathrm{id}_{\mathbf{M}(\mathrm{F})}):
\Phi\circ\mathbf{M}(\mathrm{F})\to \mathbf{M}(\mathrm{F})\circ\Phi.
\end{split}
\end{equation}
Now we claim that $(\Phi,\zeta)$ is an endomorphism of the $2$-representation $\mathbf{M}$.
Commutativity of \eqref{eq371} gives $(\mathbf{M}(\alpha)\circ_{\mathrm{h}}\mathrm{id}_{\Phi})\circ_{\mathrm{v}}\zeta_{\mathrm{F}}=
\zeta_{\mathrm{G}}\circ_{\mathrm{v}}(\mathrm{id}_{\Phi}\circ_{\mathrm{h}}\mathbf{M}(\alpha))$.
We are hence left to check the equality
\begin{equation}\label{eqnna21}
\zeta_{\mathrm{F}\circ \mathrm{G}}=(\mathrm{id}_{\mathbf{M}(\mathrm{F})}\circ_{\mathrm{h}}\zeta_{\mathrm{G}})
\circ_{\mathrm{v}}(\zeta_{\mathrm{F}}\circ_{\mathrm{h}} \mathrm{id}_{\mathbf{M}(\mathrm{G})}).
\end{equation}
Here, by definition, we have
\begin{equation*}
\begin{split}
\mathrm{id}_{\mathbf{M}(\mathrm{F})}\circ_{\mathrm{h}}\zeta_{\mathrm{G}}&=(\mathrm{id}_{\mathbf{M}(\mathrm{F})}\circ_{\mathrm{h}}
\mathrm{id}_{\mathbf{M}(\mathrm{G})}\circ_{\mathrm{h}}\theta)
\circ_{\mathrm{v}}(\mathrm{id}_{\mathbf{M}(\mathrm{F})}\circ_{\mathrm{h}}\nu\circ_{\mathrm{h}}\mathrm{id}_{\mathbf{M}(\mathrm{G})})\\
&=(\mathrm{id}_{\mathbf{M}(\mathrm{F\circ G})}\circ_{\mathrm{h}}\theta)
\circ_{\mathrm{v}}(\mathrm{id}_{\mathbf{M}(\mathrm{F})}\circ_{\mathrm{h}}\nu\circ_{\mathrm{h}}\mathrm{id}_{\mathbf{M}(\mathrm{G})})
\end{split}
\end{equation*}
and
\begin{equation*}
\begin{split}
\zeta_{\mathrm{F}}\circ_{\mathrm{h}} \mathrm{id}_{\mathbf{M}(\mathrm{G})}
&=(\mathrm{id}_{\mathbf{M}(\mathrm{F})}\circ_{\mathrm{h}}\theta\circ_{\mathrm{h}} \mathrm{id}_{\mathbf{M}(\mathrm{G})})
\circ_{\mathrm{v}}(\nu\circ_{\mathrm{h}}\mathrm{id}_{\mathbf{M}(\mathrm{F})}\circ_{\mathrm{h}} \mathrm{id}_{\mathbf{M}(\mathrm{G})})\\
&=(\mathrm{id}_{\mathbf{M}(\mathrm{F})}\circ_{\mathrm{h}}\theta\circ_{\mathrm{h}} \mathrm{id}_{\mathbf{M}(\mathrm{G})})
\circ_{\mathrm{v}}(\nu\circ_{\mathrm{h}}\mathrm{id}_{\mathbf{M}(\mathrm{F}\circ\mathrm{G})}).
\end{split}
\end{equation*}
Now \eqref{eqnna21} follows from the fact that $\nu\theta=\mathrm{Id}_{\mathbf{M}(\mathtt{i})}$.
The proof is complete.
\end{proof}

\begin{remark}\label{rem373}
{\hspace{1mm}}
{\em
\begin{enumerate}[$($i$)$]
\item\label{rem373.1} The natural isomorphism $\theta:\mathrm{Id}_{\mathbf{M}(\mathtt{i})} \to \Phi$ defines
a modification from $(\mathrm{Id}_{\mathbf{M}},\eta)$ to $(\Phi,\zeta)$
whose inverse is given by $\nu:\Phi\to \mathrm{Id}_{\mathbf{M}(\mathtt{i})}$.
Indeed, for any $1$-morphisms $\mathrm{F},\mathrm{G}$ and any $2$-morphism $\alpha:\mathrm{F}\to\mathrm{G}$, we have
\begin{equation*}
\begin{split}
(\mathbf{M}(\alpha)\circ_{\mathrm{h}}\theta)\circ_{\mathrm{v}}\eta_{\mathrm{F}}
&=\mathbf{M}(\alpha)\circ_{\mathrm{h}}\theta\\
&=(\mathrm{id}_{\mathbf{M}(\mathrm{G})}\circ_{\mathrm{h}}\theta)
\circ_{\mathrm{v}}(\mathbf{M}(\alpha)\circ_{\mathrm{h}}\mathrm{Id}_{\mathbf{M}(\mathtt{i})})\\
&=(\mathrm{id}_{\mathbf{M}(\mathrm{G})}\circ_{\mathrm{h}}\theta)
\circ_{\mathrm{v}}(\mathrm{Id}_{\mathbf{M}(\mathtt{i})}\circ_{\mathrm{h}}\mathbf{M}(\alpha))\\
&=(\mathrm{id}_{\mathbf{M}(\mathrm{G})}\circ_{\mathrm{h}}\theta)
\circ_{\mathrm{v}}(\nu\circ_{\mathrm{h}}\mathrm{id}_{\mathbf{M}(\mathrm{G})})\circ_{\mathrm{v}}(\theta\circ_{\mathrm{h}}\mathbf{M}(\alpha))\\
&=\zeta_{\mathrm{G}}
\circ_{\mathrm{v}}(\theta\circ_{\mathrm{h}}\mathbf{M}(\alpha)),
\end{split}
\end{equation*}\vspace{-3mm}
\item\label{rem373.2}
Any invertible modification $\theta$ from $(\mathrm{Id}_{\mathbf{M}},\eta)$ to some $(\Phi,\zeta)\in \mathrm{End}_{\ccC\text{-afmod}}(\mathbf{M})$ defines a natural isomorphism
from $\mathrm{Id}_{\mathbf{M}(\mathtt{i})}$ to $\Phi$.
Moreover, from the fact that $(\mathbf{M}(\mathrm{id}_{\mathrm{F}})\circ_{\mathrm{h}}\theta)\circ_{\mathrm{v}}\eta_{\mathrm{F}}=\zeta_{\mathrm{F}}
\circ_{\mathrm{v}}(\theta\circ_{\mathrm{h}}\mathbf{M}(\mathrm{id}_{\mathrm{F}}))$,
it follows that each $\zeta_{\mathrm{F}}$ is uniquely defined by \eqref{eq335} with $\nu:=\theta^{\mone}$.
\end{enumerate}}
\end{remark}

\subsection{Inductive limit construction for $2$-representations}\label{s4.33}
Assume that we are in the same setup as in Subsection~\ref{s3.7}. For any finitary $2$-representation $\mathbf{M}$ of $\cC$,
we denote by $\overline{\mathbf{M}}^{\text{\,pr}}$ the $2$-subrepresentation of $\overline{\mathbf{M}}$ with the action of $\cC$ restricted to the category $\overline{\mathbf{M}}^{\text{\,pr}}(\mathtt{i})$ consisting of projective objects in $\overline{\mathbf{M}}(\mathtt{i})$.
There exists a strict $2$-natural transformation $\Upsilon\colon\mathbf{M}\to \overline{\mathbf{M}}^{\text{\,pr}}$ given by sending an object $X$ to the diagram $0\to X$
with the obvious assignment on morphisms.
Similarly to \cite[Subsection~5.8]{MaMa}, we have a direct system
\begin{equation}\label{dr0}
\mathbf{M}\to \overline{\mathbf{M}}^{\text{\,pr}}\to \overline{(\overline{\mathbf{M}}^{\text{\,pr}})}^{\text{\,pr}}\to\cdots,
\end{equation}
where each arrow is given by $\Upsilon$ with $\mathbf{M}$ replaced by the starting point corresponding to this arrow.
We denote by $\underrightarrow{\mathbf{M}}$ the inductive limit of \eqref{dr0}.
This is a $2$-representation of $\cC$ and the natural embedding of $\mathbf{M}$
into $\underrightarrow{\mathbf{M}}$ is an equivalence.
Let $\mathcal{L}$ be a left cell of $\cC$ and $\mathbf{C}_{\mathcal{L}}:=\mathbf{N}_{\mathcal{L}}/\mathcal{I}_{\mathcal{L}}$ the corresponding cell $2$-representation.
By Yoneda Lemma, see \cite[Lemma~9]{MM2}, for any object $X$ in $\mathbf{M}(\mathtt{i})$
there exists a strict $2$-natural transformation $\Lambda_X\colon \mathbf{P}_{\mathtt{i}}\to \mathbf{M}$ which sends $\mathbbm{1}_{\mathtt{i}}$
to $X$ and, moreover, any morphism $f\colon X\to Y$ in $\mathbf{M}(\mathtt{i})$
extends to a modification
$\theta_f\colon\Lambda_X\to \Lambda_Y$. If $\mathcal{I}_{\mathcal{L}}$ annihilates $X$,
then $\Lambda_X$ induces a strict $2$-natural transformation $\Lambda'_X$ from $\mathbf{C}_{\mathcal{L}}$ to $\mathbf{M}$.
Indeed, we have the following commutative diagram
\begin{displaymath}
\xymatrix{\mathbf{N}_{\mathcal{L}}\ar@{^{(}->}[rr]^{\Xi}\ar@{->>}[d]_{\Pi}
&&\mathbf{P}_{\mathtt{i}}\ar[rr]^{\Lambda_X}&&\mathbf{M}.\\
\mathbf{C}_{\mathcal{L}}\ar[urrrr]_{\Lambda'_X}}
\end{displaymath}
If $\mathcal{I}_{\mathcal{L}}$ also annihilates $Y$, then $\Lambda_Y$ 
gives rise to a strict $2$-natural transformation $\Lambda'_Y$ from $\mathbf{C}_{\mathcal{L}}$ to $\mathbf{M}$
such that  $\Lambda_Y\Xi=\Lambda'_Y\Pi$.
Due to surjectivity of $\Pi$,
the modification $\theta_f\circ_{\mathrm{h}}\mathrm{id}_{\Xi}$
from $\Lambda_X\Xi=\Lambda'_X\Pi$ to $\Lambda_Y\Xi=\Lambda'_Y\Pi$
induces a modification $\theta'_f$ from
$\Lambda'_X$ to $\Lambda'_Y$ in $\mathrm{Hom}_{\ccC\text{-afmod}}(\mathbf{C}_{\mathcal{L}}, \mathbf{M})$.
By functoriality of the abelianization, via the limiting construction  \eqref{dr0} we thus obtain two
$2$-natural transformations $\underrightarrow{\Lambda'_X}, \underrightarrow{\Lambda'_Y}
\in \mathrm{Hom}_{\ccC\text{-afmod}}(\underrightarrow{\mathbf{C}_{\mathcal{L}}}, \underrightarrow{\mathbf{M}})$
and the modification
$\underrightarrow{\theta'}\colon \underrightarrow{\Lambda'_X}\to \underrightarrow{\Lambda'_Y}$.

\subsection{Symmetries of the cell $2$-representation}\label{s7.3}

\begin{lemma}\label{lem72}
The annihilators in $\cC$ of $L_\mathrm{F}$ and $L_\mathrm{G}$ coincide.
\end{lemma}

\begin{proof}
Since $\mathrm{G}$ is the Duflo involution in $\mathcal{H}$, it follows from \cite[Subsection~6.5]{MM2} 
that the annihilator of $L_\mathrm{F}$ is contained in the annihilator of $L_\mathrm{G}$
(as the latter is a certain unique maximal left ideal by \cite[Proposition~21]{MM2}). Furthermore, the evaluation
at $L_\mathrm{G}$, inside the abelianized cell $2$-representation,
of $\mathrm{Hom}_{\ccC}(\mathrm{H}_1,\mathrm{H}_2)$ is full for all
$\mathrm{H}_1,\mathrm{H}_2\in\{\mathrm{F},\mathrm{G}\}$ by \cite[Subsection~6.5]{MM2}.

Hence, if the annihilator of $L_\mathrm{F}$ were strictly contained in the annihilator of $L_\mathrm{G}$,
the dimension of the endomorphism space (in the cell $2$-representation) of $(\mathrm{F}\oplus \mathrm{G})L_\mathrm{F}$
would be strictly bigger than the dimension of the endomorphism space of
$(\mathrm{F}\oplus \mathrm{G})L_\mathrm{G}$. However, from Subsection~\ref{s7.2} we know
that $(\mathrm{F}\oplus \mathrm{G})L_\mathrm{F}\cong (\mathrm{F}\oplus \mathrm{G})L_\mathrm{G}$.
The claim follows.
\end{proof}

On the one hand, by \cite[Lemma~9]{MM2},
sending $\mathbbm{1}_{\mathtt{i}}$
to $L_{\mathrm{F}}$ extends to a strict $2$-natural transformation
$\Phi\colon \mathbf{P}_{\mathtt{i}}\to \overline{\mathbf{C}_{\mathcal{H}}}$.
By Lemma~\ref{lem72}, we know that $\Phi\Xi$ factors through
$\mathbf{C}_{\mathcal{H}}$ and obtain
a strict $2$-natural transformation $\Phi'$ from $\mathbf{C}_{\mathcal{H}}$ to $\overline{\mathbf{C}_{\mathcal{H}}}$. Note that $\Phi$ sends both $\mathrm{F}$ and $\mathrm{G}$ to projective objects in $\overline{\mathbf{C}_{\mathcal{H}}}(\mathtt{i})$.
Therefore $\Phi'$ is also a strict $2$-natural transformation from $\mathbf{C}_{\mathcal{H}}$ to $\overline{\mathbf{C}_{\mathcal{H}}}^{\text{\,pr}}$
and we have the following commutative diagram:
\begin{equation}\label{eq73}
\xymatrix{\mathbf{N}_{\mathcal{L}}\ar@{^{(}->}[rr]^{\Xi}\ar[drr]^{\Phi\Xi}\ar@{->>}[dd]_{\Pi}
&&\mathbf{P}_{\mathtt{i}}\ar[rr]^{\Phi}&&
\overline{\mathbf{C}_{\mathcal{H}}}\ar[rr]^{\overline{\Phi'}}&&
\overline{(\overline{\mathbf{C}_{\mathcal{H}}}^{\text{\,pr}})}\\
&&\overline{\mathbf{C}_{\mathcal{H}}}^{\text{\,pr}}\ar@{^{(}->}[urr]
&&\overline{(\overline{\mathbf{C}_{\mathcal{H}}}^{\text{\,pr}})}^{\text{\,pr}}
\ar@{^{(}->}[urr]\\
\mathbf{C}_{\mathcal{H}}\ar[urr]^{\Phi'}
\ar@/_1.5pc/[urrrr]^{\tilde{\Phi}}\ar@/_2.0pc/[uurrrr]^{\Phi'}
\ar@/_3.5pc/[uurrrrrr]^{\tilde{\Phi}}&&&&}
\end{equation}
Applying the 
procedure in Subsection~\ref{s4.33} to $\Phi'\colon\mathbf{C}_{\mathcal{H}}\to\overline{\mathbf{C}_{\mathcal{H}}}^{\text{\,pr}}$, 
we obtain a strict $2$-natural transformation
$\underrightarrow{\Phi'}$ in $\mathrm{End}_{\ccC\text{-afmod}}(\underrightarrow{\mathbf{C}_{\mathcal{H}}})$
which swaps the isomorphism classes of indecomposable projectives.

On the other hand, sending $\mathbbm{1}_{\mathtt{i}}$
to $L_{\mathrm{G}}$ extends to a strict $2$-natural transformation
$\Psi\colon \mathbf{P}_{\mathtt{i}}\to \overline{\mathbf{C}_{\mathcal{H}}}$,
and the latter induces a strict 2-natural transformation $\Psi'$
from $\mathbf{C}_{\mathcal{H}}$ to $\overline{\mathbf{C}_{\mathcal{H}}}$ (which factors through $\overline{\mathbf{C}_{\mathcal{H}}}^{\text{\,pr}}$).
For $\Psi$, we have similar diagram as \eqref{eq73}.
Note that $\overline{\Psi'}\Psi(\mathbbm{1}_{\mathtt{i}})=L_{\mathrm{G}}$ and $\overline{\Phi'}(L_{\mathrm{F}})\cong L_{\mathrm{G}}$.
For a fixed isomorphism $\alpha\colon L_{\mathrm{G}}\to\overline{\Phi'}(L_{\mathrm{F}})$,
by Subsection~\ref{s4.33} 
there exists an invertible modification $\vartheta$ from $\tilde{\Psi}=\overline{\Psi'}\Psi'$ to $\tilde{\Phi}=\overline{\Phi'}\Phi'$
(here both equalities are in $\mathrm{Hom}_{\ccC\text{-afmod}}(\mathbf{C}_{\mathcal{H}}, \overline{(\overline{\mathbf{C}_{\mathcal{H}}}^{\text{\,pr}})}^{\text{\,pr}})$.
Note that the limiting construction  \eqref{dr0} applied to $\tilde{\Psi}$
gives a functor isomorphic to $\mathrm{Id}_{\underrightarrow{\mathbf{C}_{\mathcal{H}}}}$.
Using Subsections~\ref{s3.7} and \ref{s4.33}, we thus get an invertible modification $\underrightarrow{\vartheta}$ from $\mathrm{Id}_{\underrightarrow{\mathbf{C}_{\mathcal{H}}}}$ to $(\underrightarrow{\Phi'})^2$.
Following \cite[Lemma~18]{MaMa} and the proof of \cite[Proposition~19]{MaMa},
we obtain that
\begin{enumerate}[$($a$)$]
\item\label{pt1} for any $\upsilon\in \mathrm{Hom}_{\ccC\text{-afmod}}(\mathrm{Id}_{\underrightarrow{\mathbf{C}_{\mathcal{H}}}}, (\underrightarrow{\Phi'})^2)$, we have $\mathrm{id}_{(\underrightarrow{\Phi'})^2}\circ_{\mathrm{h}}\upsilon
=\upsilon\circ_{\mathrm{h}}\mathrm{id}_{(\underrightarrow{\Phi'})^2}$;
\item\label{pt2} there exists an invertible modification $\upsilon\in \mathrm{Hom}_{\ccC\text{-afmod}}(\mathrm{Id}_{\underrightarrow{\mathbf{C}_{\mathcal{H}}}}, (\underrightarrow{\Phi'})^2)$  such that we have either $\mathrm{id}_{\underrightarrow{\Phi'}}\circ_{\mathrm{h}}\upsilon
=\upsilon\circ_{\mathrm{h}}\mathrm{id}_{\underrightarrow{\Phi'}}$ or $\mathrm{id}_{\underrightarrow{\Phi'}}\circ_{\mathrm{h}}\upsilon=
-\upsilon\circ_{\mathrm{h}}\mathrm{id}_{\underrightarrow{\Phi'}}$.
\end{enumerate}

Note that $(\underrightarrow{\Phi'})^2$ preserves the isomorphism classes of projectives
and hence defines an auto-equivalence of
$\underrightarrow{\mathbf{C}_{\mathcal{H}}}$ which is isomorphic to the identity.
Therefore $\underrightarrow{\Phi'}$ induces an automorphism $\varphi$ of
$A$ and such that $\varphi^2$, corresponding to $(\underrightarrow{\Phi'})^2$, is an inner
automorphism of $A$, cf. \cite[Lemma~1.10.9]{Zi}. Assume that
the inner automorphism $\varphi^2$ is of the form $x\mapsto axa^{\mone}$, where $x\in A$,
for some fixed invertible element $a\in A$.
Similarly to the paragraph before \cite[Proposition~39]{KMMZ},
there exists an element $b\in A$ which is a polynomial in $a^{\mone}$ and such that
$b^2=a^{\mone}$. Let $\sigma$ be the inner automorphism of $A$ given by
$x\mapsto bxb^{\mone}$, for $x\in A$. 

\begin{lemma}\label{nnlem75}
We have $(\sigma\varphi)^4=\mathrm{id}_A$. 
\end{lemma}

\begin{proof}
The obvious fact that $\varphi$ and $\varphi^2$ commute is equivalent to the 
requirement that $t:=\varphi(a^{\mone})a$ belongs to the center of $A$. 
We have
\begin{displaymath}
\varphi(t)=\varphi^2(a^{\mone})\varphi(a)=aa^{\mone}a^{\mone}\varphi(a)
=a^{\mone}\varphi(a)=t^{\mone}.
\end{displaymath}
Therefore $\varphi^2(t)=t=a\varphi(a^{\mone})aa^{\mone}=a\varphi(a^{\mone})$, 
which implies that $\varphi(a^{\mone})$ and $a$ commute.
Consequently, $\varphi(a^{\mone})$ and $a^{\mone}$ commute.
This implies that any polynomial in $\varphi(a^{\mone})$ commutes with
any polynomial in $a^{\mone}$. Therefore $\varphi(b)$ and $b$ commute
and thus $\varphi(b^{\mone})$ and $b$ commute as well. Hence the elements $a$, $a^{\mone}$,
$b$, $b^{\mone}$, $\varphi(a)$, $\varphi(a^{\mone})$, $\varphi(b)$, $\varphi(b^{\mone})$ all commute.

A direct computation shows that the action of  $(\sigma\varphi)^4$ on $A$ is given by conjugation with
\begin{displaymath}
b\varphi(b)aba^{-1}\varphi(a)\varphi(b)\varphi(a^{\mone})a^2. 
\end{displaymath}
Using commutativity of the factors, this reduces to $a\varphi(a^{\mone})$ which is central. The claim follows.
\end{proof}

The functor of twisting $A$-modules by $\sigma$ is isomorphic to the identity functor as $\sigma$ is inner.
By Lemma~\ref{lem333}, the functor of twisting by $\sigma$ gives rise to an endomorphism
$\Sigma$ of $\underrightarrow{\mathbf{C}_{\mathcal{H}}}$ which preserves the isomorphism classes of projectives.
Then the $2$-natural transformation 
$\Omega:=\Sigma\underrightarrow{\Phi'}\in \mathrm{End}_{\ccC\text{-afmod}}(\underrightarrow{\mathbf{C}_{\mathcal{H}}})$
induces an automorphism on $A$ given by $\sigma\varphi$. We denote this automorphism by $\iota$.

\begin{example}\label{exnew756}
{\em 
Let $A$ be the quotient of the path algebra of the quiver
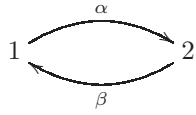
\begin{displaymath}
\xymatrix{
1\ar@/^1pc/[rr]^{\alpha}&&2\ar@/^1pc/[ll]^{\beta}
} 
\end{displaymath}
modulo the relations $\alpha\beta=\beta\alpha=0$. Let $\varphi$ be the automorphism of $A$ defined by
$\varphi(e_1)=e_2$, $\varphi(e_2)=e_1$, $\varphi(\alpha)=-\beta$ and $\varphi(\beta)=\alpha$. 
Then $\varphi^4=\mathrm{id}_A$ but $\varphi^2\neq\mathrm{id}_A$. In fact, $\varphi^2$ is conjugation by $a=e_1-e_2$. 
Note that the element $\varphi(a^{\mone})a=-e_1-e_2$ is central. This example shows $\underrightarrow{\Phi'}$
does not necessarily correspond to an automorphism of order $2$.
}
\end{example}

\subsection{Connection to $\cG_{A}$}\label{s7.4}

Set $G$ to be the cyclic group generated by $\iota$ 
(note that $|G|=2$ or $|G|=4$) and consider the fiat $2$-category $\cG_{A}$,
where $A$ is the underlying algebra of $\mathbf{C}_{\mathcal{H}}$.
Let $\cH_A$ denote the full and faithful $2$-subcategory of $\cG_A$ 
generated by $(A, \tilde{\pi}_{1_{\hat{G}}})$ and
$1$-morphisms in the two-sided cell $\mathcal{J}_0$, referring to
Subsection~\ref{s3.3} and Subsection~\ref{s3.4} for notation.

\begin{theorem}\label{prop41}
If $\cC$ is $\mathcal{H}$-simple, then $\cC$ is biequivalent to a $2$-subcategory of $\cH_A$.
\end{theorem}

\begin{proof}
As mentioned above, $\underrightarrow{\mathbf{C}_{\mathcal{H}}}$ is equivalent to the cell $2$-representation $\mathbf{C}_{\mathcal{H}}$.
As $\cC$ is $\mathcal{H}$-simple, the $2$-representation $\underrightarrow{\mathbf{C}_{\mathcal{H}}}$ gives a faithful $2$-functor from $\cC$ to $\cC_A$.
Note that the $1$-morphisms $\mathrm{F}$ and $\mathrm{G}$ are represented,
respectively, by $X, Y$ in \eqref{eq13}  under the $2$-functor $\underrightarrow{\mathbf{C}_{\mathcal{H}}}$.
Assume that the family of natural isomorphisms 
$$
\eta:=\{\eta_{\mathrm{H}}\colon 
\Omega\circ\underrightarrow{\mathbf{C}_{\mathcal{H}}}(\mathrm{H})\to \underrightarrow{\mathbf{C}_{\mathcal{H}}}(\mathrm{H})\circ\Omega,\, 
\mathrm{H}\in\cC(\mathtt{i},\mathtt{i})\}
$$
is the data associated to the $2$-natural transformation $\Omega\in\mathrm{End}_{\ccC\text{-afmod}}(\underrightarrow{\mathbf{C}_{\mathcal{H}}})$ constructed above.
Thus, for any $2$-morphism $\alpha: \mathrm{H}\to\mathrm{K}$ in $\cC$,
we have 
$$
(\underrightarrow{\mathbf{C}_{\mathcal{H}}}(\alpha)\circ_{\mathrm{h}}\Omega)\circ_{\mathrm{v}}\eta_{\mathrm{H}}
=\eta_{\mathrm{K}}\circ_{\mathrm{v}}(\Omega\circ_{\mathrm{h}}\underrightarrow{\mathbf{C}_{\mathcal{H}}}(\alpha)).
$$
Due to the fact that $\Omega$ swaps the isomorphism classes of projectives in $\underrightarrow{\mathbf{C}_{\mathcal{H}}}(\mathtt{i})$,
all $A$-$A$-bimodule homomorphisms corresponding to non-zero $2$-morphisms in $\cC$ are symmetric in the way which makes sense.
\end{proof}

\vspace{2mm}

\noindent
Vo.~Ma.: Department of Mathematics, Uppsala University, Box. 480,
SE-75106, Uppsala, SWEDEN, email: {\tt mazor\symbol{64}math.uu.se}

\noindent
Va.~Mi.: School of Mathematics, University of East Anglia,
Norwich NR4 7TJ, UK,  {\tt v.miemietz\symbol{64}uea.ac.uk}

\noindent
X.~Z.: Department of Mathematics, Uppsala University, Box. 480,
SE-75106, Uppsala, SWEDEN, email: {\tt xiaoting.zhang\symbol{64}math.uu.se}

\end{document}